\documentclass[12pt]{amsart}
\usepackage[utf8]{inputenc} 
\usepackage[T1]{fontenc}
\usepackage{color}
\usepackage{comment}
\usepackage{times}
\usepackage[hidelinks]{hyperref}
\usepackage{enumerate,latexsym}
\usepackage{latexsym}
\usepackage{subcaption}

\usepackage{fullpage}

\usepackage{amsmath,amsthm,amsfonts,amssymb}
\usepackage{graphicx}
\usepackage{dsfont}

\usepackage{tikz}
\usetikzlibrary{arrows.meta}
\usetikzlibrary{knots}
\usetikzlibrary{hobby}
\usetikzlibrary{arrows,decorations.markings}

\usetikzlibrary{fadings}

\makeatletter
\def\namedlabel#1#2{\begingroup
 #2%
 \def\@currentlabel{#2}%
 \phantomsection\label{#1}\endgroup
}
\makeatother

\usepackage[lite]{amsrefs}

\renewcommand{\PrintDOI}[1]{\href{http://dx.doi.org/\detokenize{#1}}{doi: \detokenize{#1}}%
	\IfEmptyBibField{pages}{, (to appear in print)}{}}

\theoremstyle{plain}
\newtheorem*{theorem*}{Theorem}
\newtheorem*{thmex*}{Theorem~\ref{example}}
\newtheorem*{thmasymp*}{Theorem~\ref{thmAsymp}}
\newtheorem{theorem}{Theorem}[section]

\newtheorem{proposition}[theorem]{Proposition}
\newtheorem{remark}[theorem]{Remark}

\newtheorem{example}[theorem]{Example}

\theoremstyle{definition}
\newtheorem{definition}[theorem]{Definition}

\newcommand{\ben}{\begin{enumerate}}
\newcommand{\een}{\end{enumerate}}

\newcommand{\ed}{\end{document}}

\definecolor{rrr}{rgb}{.9,0,.1}

\definecolor{rr}{rgb}{.8,0,.3}

\usepackage{pdfsync}
\graphicspath{ {./images/} }

\title[Generating Set of Reidemeister Moves of Oriented Stuck Links and Quandles]{RNA foldings and Stuck Knots}

\author[J. Ceniceros]{Jose Ceniceros}
\address{Hamilton College, Clinton, NY, USA}
\email{jcenicer@hamilton.edu}

\author[M. Elhamadi]{Mohamed Elhamdadi}
\address{University of South Florida, Tampa, Florida, USA}
\email{emohamed@usf.edu}

\author[J. Komissar]{Josef Komissar}
\address{Hamilton College, Clinton, NY, USA}
\email{jkomissa@hamilton.edu}

\author[H. Lahrani]{Hitakshi Lahrani}
\address{University of South Florida, Tampa, Florida, USA}
\email{lahrani@usf.edu}

\begin{document}

\maketitle

\begin{abstract}
We study RNA foldings and investigate their topology using a combination of knot theory and embedded rigid vertex graphs. Knot theory has been helpful in modeling biomolecules, but classical knots place emphasis on a biomolecule's entanglement while ignoring their intrachain interactions. We remedy this by using stuck knots and links, which provide a way to emphasize both their entanglement and intrachain interactions. We first give a generating set of the oriented stuck Reidemeister moves for oriented stuck links. We then introduce an algebraic structure to axiomatize the oriented stuck Reidemeister moves. Using this algebraic structure, we define a coloring counting invariant of stuck links and provide explicit computations of the invariant. Lastly, we compute the counting invariant for arc diagrams of RNA foldings through the use of stuck link diagrams.

\end{abstract}

\section{Introduction}\label{intro}

We introduce an algebraic structure in order to study RNA folding. These structures will be used to investigate and classify the topology of RNA foldings, and the classification is highly relevant to molecular biology. The structures we propose are related to stuck links introduced in \cite{B}. The main motivating factor for utilizing stuck knots and links is their application to modeling biomolecules. Classical knot theory has been used to model biomolecules with limited success.  For example, in \cite{SRMOS}, an overwhelming number of the proteins in the Protein Data Bank were classified into one specific isotopy class, the unknot. One potential reason for the limited success of classical knot theory is that the emphasis is placed on the knottiness of the biomolecule while completely ignoring any intrachain interactions or bonds as noted in \cite{MWT}. In order to address these limitations graphical models of folded molecules lead to the study of embedded rigid vertex graphs in general and singular knots \cite{V} in particular. We will consider a generalization of singular links called stuck links since they are uniquely equipped to model the entanglement and intrachain interactions of a biomolecule as described in \cite{B}.

Specifically, this research aims to develop tools to classify arc diagrams of RNA foldings by developing an algebraic structure and invariant of oriented stuck links. In order to define our invariant, we first introduce a generating set of the Reidemeister type moves for oriented stuck links. Using the generating set of stuck Reidemeister moves and motivated by the effectiveness of quandles and singqandles in distinguishing oriented knots and singular knots and links, we define \emph{stuquandles}. We then use stuquandles to define a computable invariant of oriented stuck knots and links. 

 Equipped with the stuquandle structure and counting invariant we will shift our focus to the relationship between stuck links and RNA folding. In \cite{B}, a transformation was introduced between arc diagrams of RNA foldings defined in \cite{KM} and stuck links. Arc diagrams of RNA foldings have been studied, and polynomial invariants have been defined in order to classify them in \cites{KM,B, TLKL}. In contrast to the approaches previously taken, which depend on skein relations, we define an invariant that is easily computable with {\tt Mathematica}, {\tt Maple}, and {\tt Python}. Our invariant of RNA foldings is a consequence of the stuquandle structure and stuquandle counting invariant of stuck links. Furthermore, we show that the stuquandle counting invariant is a computable and effective invariant of RNA foldings.

The article is organized as follows. In Section~\ref{RSK}, we review the basics of the theory of stuck knots and links as introduced in \cite{B}. In Section~\ref{arcdiagrams}, we review the relationship between arc diagrams of RNA foldings and stuck links. In Section~\ref{GenSet}, we introduce a generating set of oriented Reidemeister moves for oriented stuck links. In Section~\ref{stuckq}, we introduce the notion of an oriented stuquandle and provide examples. In Section~\ref{computations}, we provide applications of oriented stuquandles to distinguish oriented stuck links. Lastly, in Section~\ref{app}, we apply oriented stuquandles and stuck knots and links to distinguish arc diagrams of RNA folding.

\section{Review of Stuck Knots and Links}\label{RSK}

In this article, we will follow the definitions and conventions of stuck knots introduced 
in \cite{B}. 
The theory of stuck knots can be thought of as a generalization of the theory of singular knots in the sense that a diagram of a stuck knot may contain classical crossings and stuck crossings. A stuck crossing is a singular crossing with additional under and over information. See Figure \ref{SX} for a singular crossing and its representation in a singular link diagram. See Figure \ref{StuckX} for the two possible ways a stuck crossing may be stuck along with their representations in an oriented stuck link diagram 

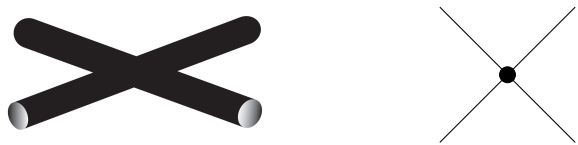
\begin{figure}[ht]
    \centering
    \includegraphics[scale=.08]{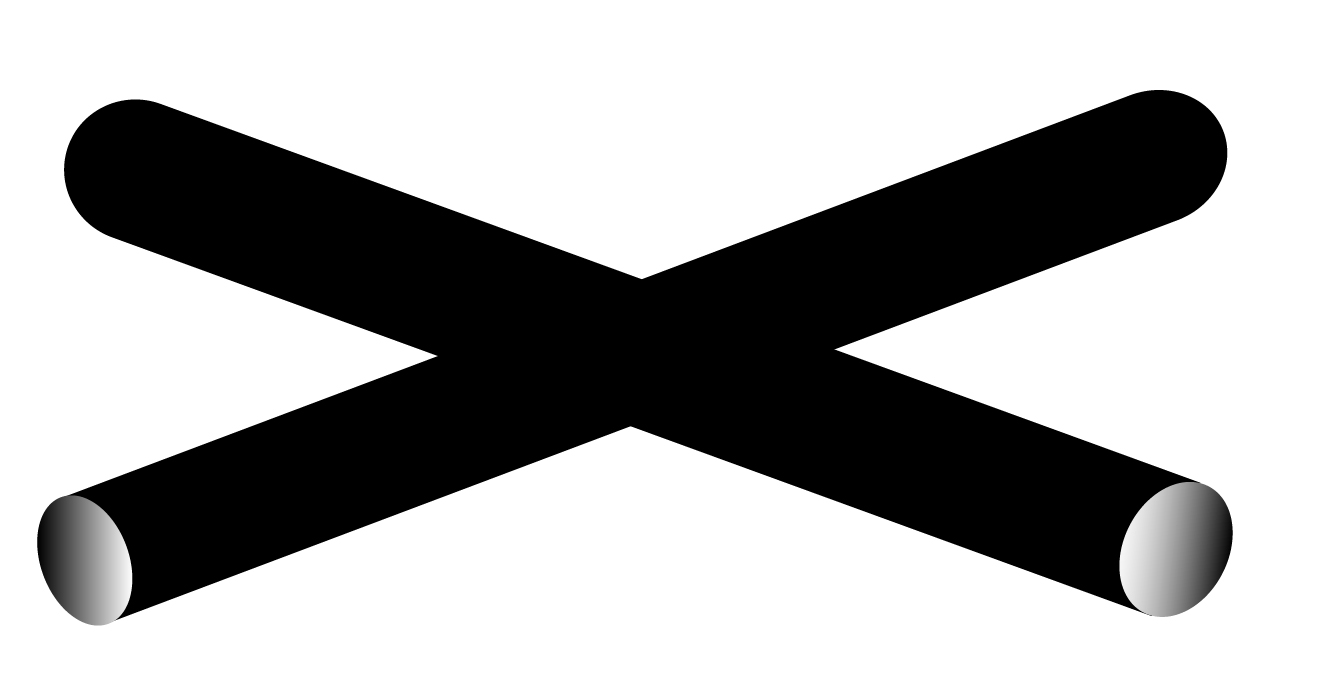}\hspace{2cm}
    \begin{tikzpicture}[use Hobby shortcut, scale=.9]
\begin{knot}[
  consider self intersections = true,
  clip width=5,
  flip crossing/.list={1,5,6,7,8}
]
\draw (-6,1)..(-4,-1);
\draw (-4,1)..(-6,-1);
\end{knot}
\node[circle,draw=black, fill=black, inner sep=0pt,minimum size=6pt] (a) at (-5,0) {};
\end{tikzpicture}
    \caption{Singular crossing in singular link and a singular crossing in a singular diagram.}
    \label{SX}
\end{figure}

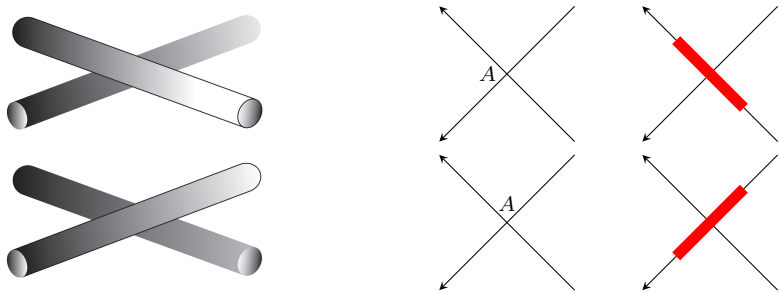
\begin{figure}
    \centering
    \includegraphics[scale=.08]{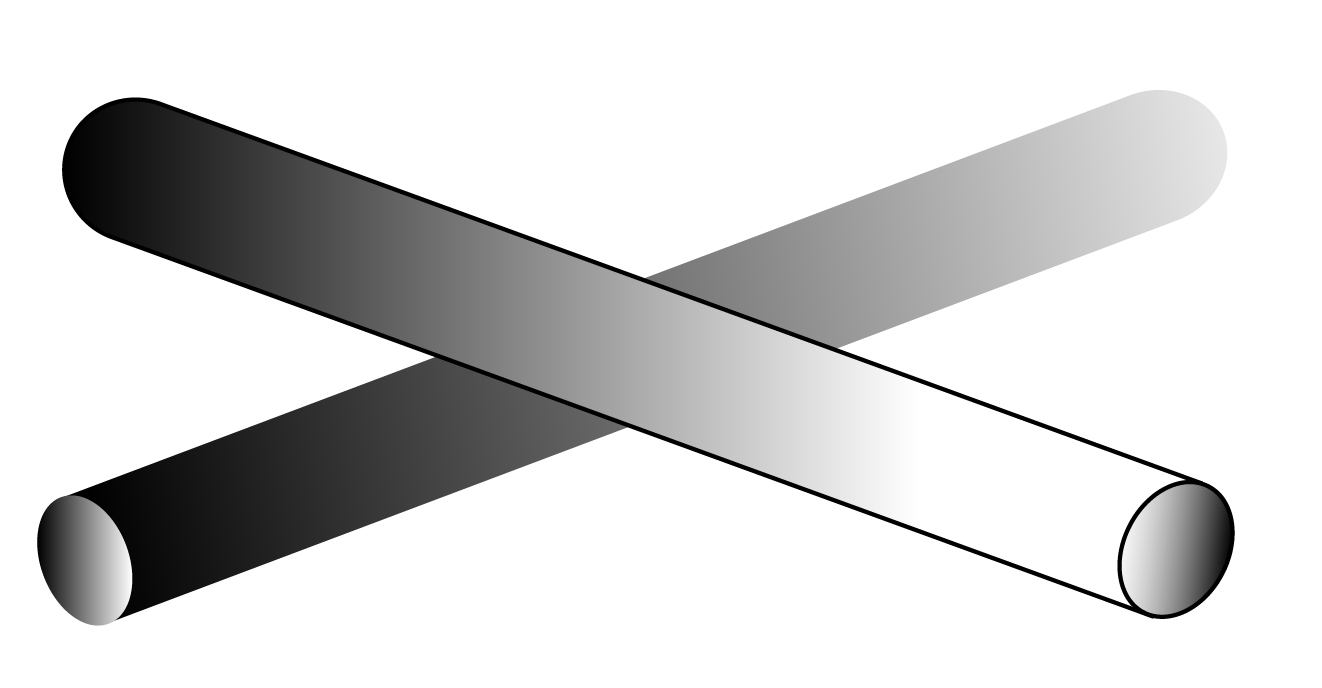}\hspace{2cm}
    \begin{tikzpicture}[use Hobby shortcut, scale=.9]
\begin{knot}[
  consider self intersections = true,
  clip width=5,
  flip crossing/.list={1,5,6,7,8}
]
\draw[stealth-] (-6,1)..(-4,-1);
\draw[-stealth] (-4,1)..(-6,-1);
\end{knot}
\node[left] at (-5,0) {\tiny$A$};

\begin{knot}[
  consider self intersections = true,
  clip width=5,
  flip crossing/.list={1,5,6,7,8}
]
\draw[stealth-] (-3,1)..(-1,-1);
\draw[-stealth] (-1,1)..(-3,-1);
\draw[line width=1.5mm,red] (-2.5,.5)..(-1.5,-.5);
\end{knot}
\end{tikzpicture}

    \includegraphics[scale=.08]{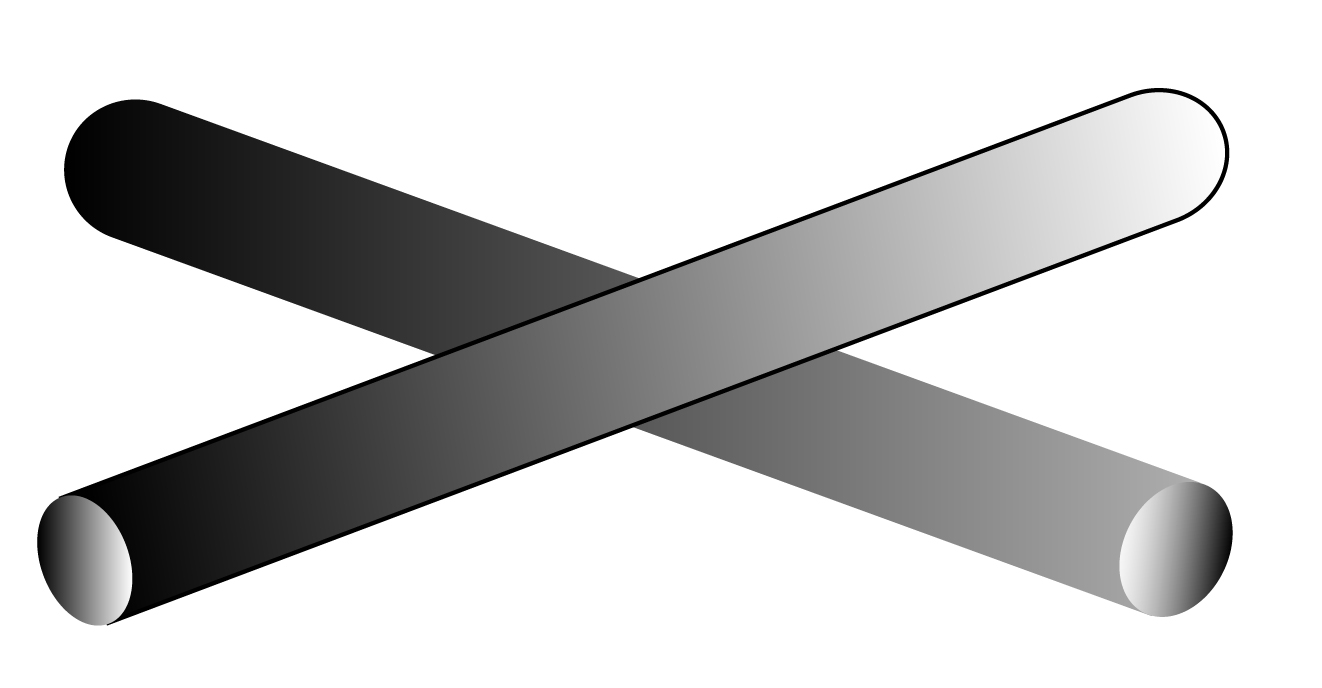}\hspace{2cm}
\begin{tikzpicture}[use Hobby shortcut, scale=.9]
\begin{knot}[
  consider self intersections = true,
  clip width=5,
  flip crossing/.list={1,5,6,7,8}
]
\draw[stealth-] (-6,1)..(-4,-1);
\draw[-stealth] (-4,1)..(-6,-1);
\end{knot}
\node[above] at (-5,0) {\tiny$A$};

\begin{knot}[
  consider self intersections = true,
  clip width=5,
  flip crossing/.list={1,5,6,7,8}
]
\draw[stealth-] (-3,1)..(-1,-1);
\draw[-stealth] (-1,1)..(-3,-1);
\draw[line width=1.5mm,red] (-2.5,-.5)..(-1.5,.5);
\end{knot}
\end{tikzpicture}
\caption{Positively stuck crossing (top) and negatively stuck crossing (bottom).}
    \label{StuckX}
\end{figure}

At a stuck crossing we will be assigning the $A$-region introduced by Kauffman in \cite{K} in order to indicate the over and under information at a stuck crossing, see middle figures in Figure \ref{StuckX}. Alternatively, we will use a thick red bar on the over strand at a stuck crossing, see right figures in Figure \ref{StuckX}. We will refer to the top crossing in Figure \ref{StuckX} as being a positive stuck crossing while the bottom crossing we will refer to as a negative stuck crossing. 

Similar to the classical case, an oriented stuck link is an equivalence class of oriented stuck link diagrams subject to oriented classical Reidemeister moves and oriented stuck Reidemeister moves. See Figures~\ref{stuck1}, \ref{stuck2} and \ref{stuck3} for the oriented stuck Reidemeister moves. 

There are some immediate invariants of stuck knots. For example, the \emph{sticking number} of a stuck knot or link is the number of stuck crossings, see \cite{B}. Additionally, there is the following variation of the sticking number.  
\begin{definition}\cite{B} \label{stickingnumber}
The \emph{signed sticking number} of an oriented stuck knot or link is the sum of signs of the stuck crossings.
\end{definition}

Examples~\ref{ex7.3} and \ref{ex8.1} show that the invariant we introduce in this article called  \emph{the stuquandle counting invariant} is stronger than the signed sticking number.

\section{RNA Folding and Stuck Knots}\label{arcdiagrams}

In this section, we will describe the connection between arc diagrams of RNA folding and stuck knots. Arc diagrams were introduced by Kauffmann and Magarshak in \cite{KM} as a combinatorial way of studying RNA folding. As described in \cite{KM}, the RNA molecule is a long chain consisting of the bases A (adenine), C (cytosine), U (uracil), and G (guanine). The pairs (A and U) and (C and G) are able of forming bonds to each other. Therefore, an RNA molecule may be represented as a linear sequence of the letters A, C, G, and U and a folding of the molecule is a possible pairing of a given sequence of bases, for a complete description and theory of arc diagrams see \cite{KM}. The following example was taken from \cite{KM}, consider the chain $\cdots CCCAAAACCCCCUUUUCCC\cdots$ with the folding given in Figure~\ref{RNAchain}.
\begin{figure}[ht]
    \centering
    \includegraphics[scale=.5]{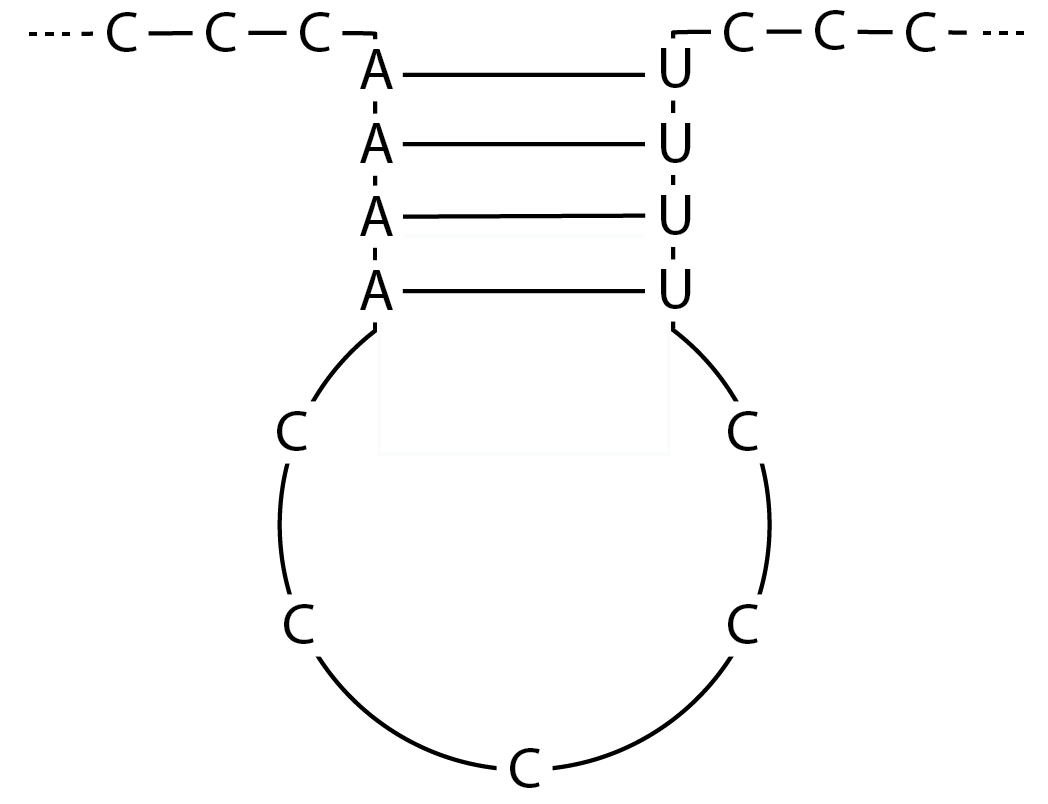}
    \caption{RNA folding}
    \label{RNAchain}
\end{figure}
\\
Kauffmann and Magarshak abstracted an RNA folding by introducing arc diagrams, see figure on the left in Figure~\ref{RNAfolding}. We will use the convention introduced in \cite{B} of a gray stripe in place of the four connecting arcs, see figure on the right of Figure~\ref{RNAfolding}.

\begin{figure}[ht]
    \centering
    \includegraphics[scale=.5]{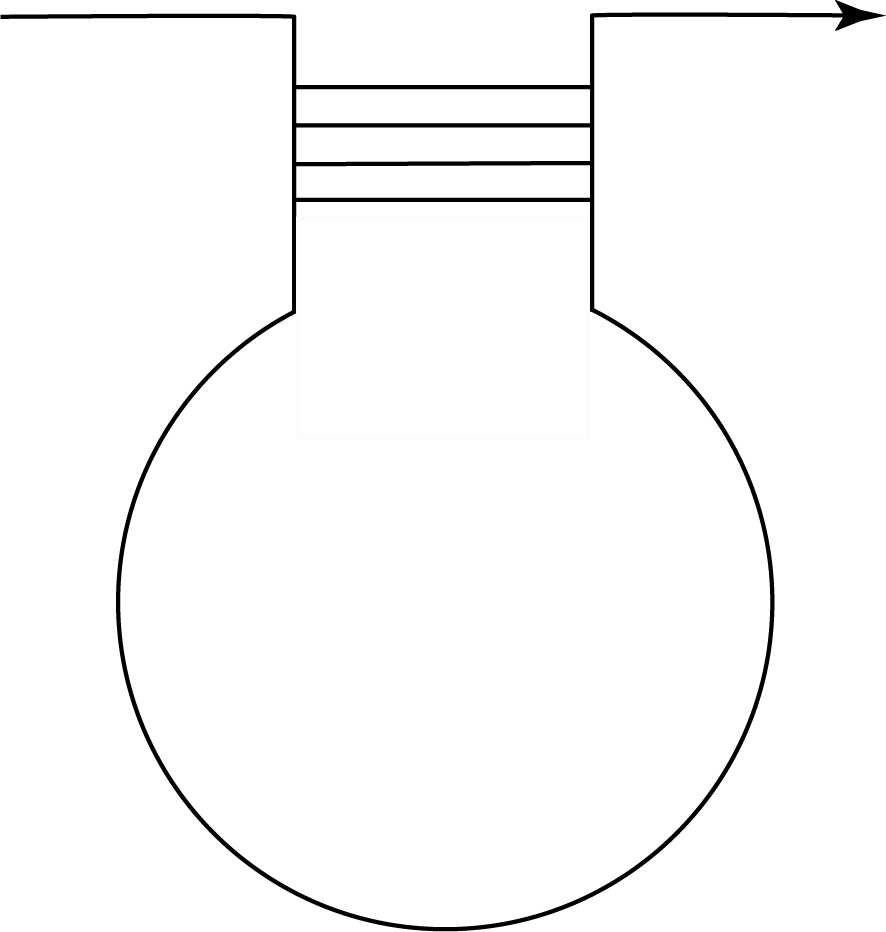}
    \hspace{1cm}
    \includegraphics[scale=.5]{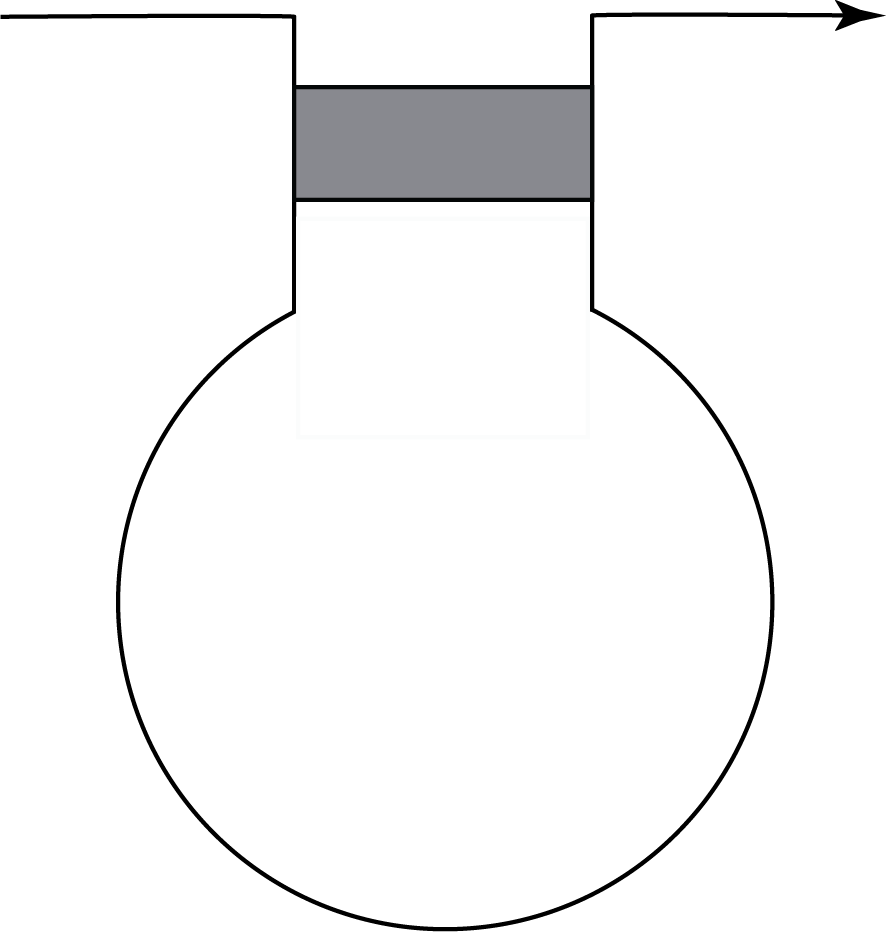}
    \caption{Arc diagram of RNA folding (left) and an arc diagram of RNA folding with gray stripe (right). }
    \label{RNAfolding}
\end{figure}
The following transformation which relates arc diagrams and stuck links was defined in \cite{B}. In order to change from an arc diagram of RNA folding to a stuck knot or link we will use the transformation, $T$, see Figure~\ref{TRNAfolding1}. 

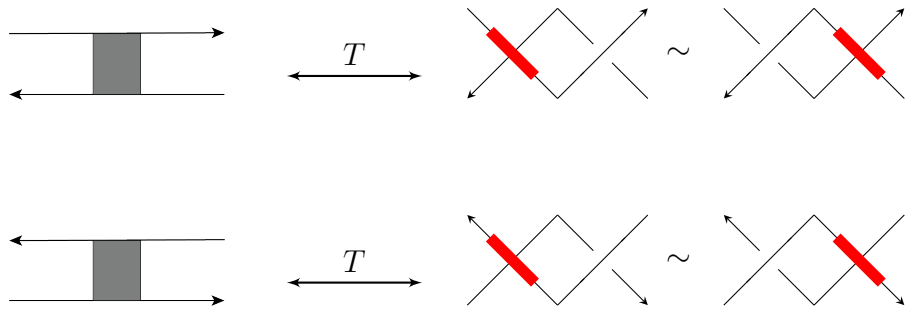
\begin{figure}[ht]
    \centering
    \includegraphics[scale=.5]{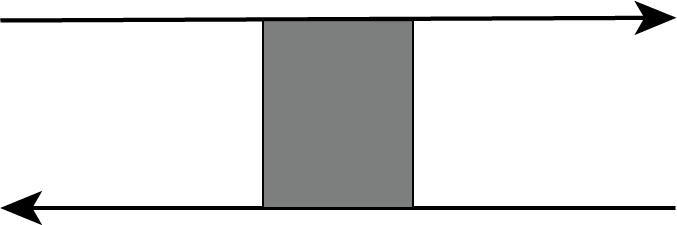}
    \hspace{.6cm}
    \begin{tikzpicture}[use Hobby shortcut, scale=.6]
\begin{knot}[
  consider self intersections = true,
  clip width=5,
  flip crossing/.list={5,6,7,8}
]
\draw (-3,1)..(-1,-1);
\strand[-stealth] (-1,-1)..(1,1);
\draw[-stealth] (-1,1)..(-3,-1);
\strand (-1,1)..(1,-1);
\draw[line width=1.5mm,red] (-2.5,.5)..(-1.5,-.5);
\draw[latex'-latex', thick](-7,-.5)..(-4,-.5);
\end{knot}
\node[above] at (-5.5,-.5) {$T$};
\end{tikzpicture}
\begin{tikzpicture}[use Hobby shortcut, scale=.6]
\begin{knot}[
  consider self intersections = true,
  clip width=5,
  flip crossing/.list={1,5,6,7,8}
]
\strand (-3,1)..(-1,-1);
\draw[-stealth] (-1,-1)..(1,1);
\strand[-stealth] (-1,1)..(-3,-1);
\draw (-1,1)..(1,-1);
\end{knot}
\draw[line width=1.5mm,red] (-.5,.5)..(.5,-.5);
\node at (-4,0) {$\sim$};
\end{tikzpicture}

    \vspace{1.5cm}
    \includegraphics[scale=.5]{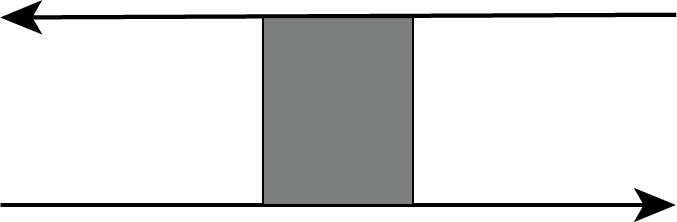}
    \hspace{.6cm}
        \begin{tikzpicture}[use Hobby shortcut, scale=.6]
\begin{knot}[
  consider self intersections = true,
  clip width=5,
  flip crossing/.list={5,6,7,8}
]
\draw[stealth-] (-3,1)..(-1,-1);
\strand(-1,-1)..(1,1);
\draw (-1,1)..(-3,-1);
\strand[-stealth] (-1,1)..(1,-1);
\draw[line width=1.5mm,red] (-2.5,.5)..(-1.5,-.5);
\draw[latex'-latex', thick](-7,-.5)..(-4,-.5);
\end{knot}
\node[above] at (-5.5,-.5) {$T$};
\end{tikzpicture}
\begin{tikzpicture}[use Hobby shortcut, scale=.6]
\begin{knot}[
  consider self intersections = true,
  clip width=5,
  flip crossing/.list={1,5,6,7,8}
]
\strand[stealth-] (-3,1)..(-1,-1);
\draw (-1,-1)..(1,1);
\strand(-1,1)..(-3,-1);
\draw[-stealth]  (-1,1)..(1,-1);
\end{knot}
\draw[line width=1.5mm,red] (-.5,.5)..(.5,-.5);
\node at (-4,0) {$\sim$};
\end{tikzpicture}
    \caption{Transformation between arc diagram and stuck diagram.}
    \label{TRNAfolding1}
\end{figure}
The same transformation maybe be used when starting with a negative stuck crossing, see Figure~\ref{TRNAfolding2}. 
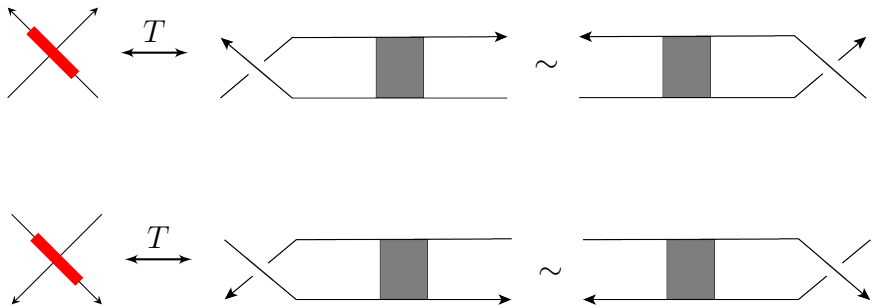
\begin{figure}[ht]
    \centering
    \begin{tikzpicture}[use Hobby shortcut, scale=.6]
\begin{knot}[
  consider self intersections = true,
  clip width=5,
  flip crossing/.list={5,6,7,8}
]
\draw[stealth-] (-3,1)..(-1,-1);
\draw[stealth-] (-1,1)..(-3,-1);
\draw[line width=1.5mm,red] (-2.5,.5)..(-1.5,-.5);
\draw[latex'-latex', thick](-.5,0)..(1,0);
\end{knot}
\node[above] at (.25,0) {$T$};
\end{tikzpicture}
\hspace{.2cm}
    \includegraphics[scale=.5]{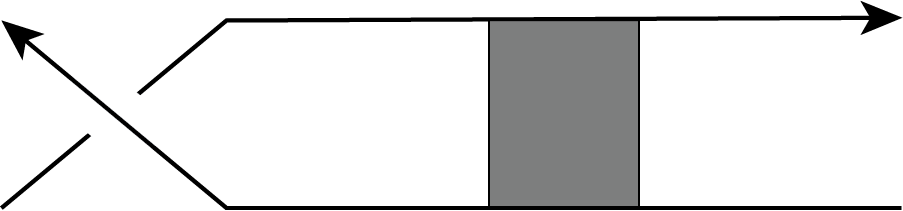}
    \put(10,10){$\sim$}
    \includegraphics[scale=.5]{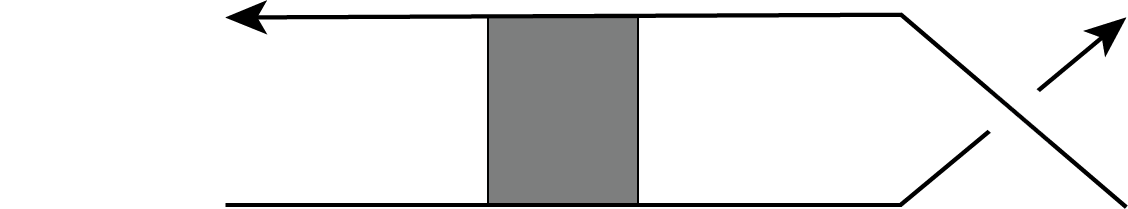}
    \vspace{1.5cm}
    
\begin{tikzpicture}[use Hobby shortcut, scale=.6]
\begin{knot}[
  consider self intersections = true,
  clip width=5,
  flip crossing/.list={5,6,7,8}
]
\draw[-stealth] (-3,1)..(-1,-1);
\draw[-stealth] (-1,1)..(-3,-1);
\draw[line width=1.5mm,red] (-2.5,.5)..(-1.5,-.5);
\draw[latex'-latex', thick](-.5,0)..(1,0);
\end{knot}
\node[above] at (.25,0) {$T$};
\end{tikzpicture}
\hspace{.2cm}    
    \includegraphics[scale=.5]{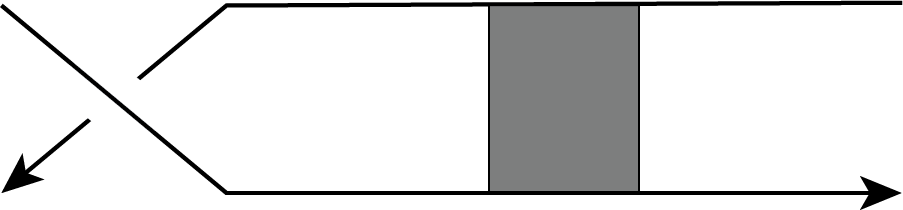}
        \put(10,10){$\sim$}
    \includegraphics[scale=.5]{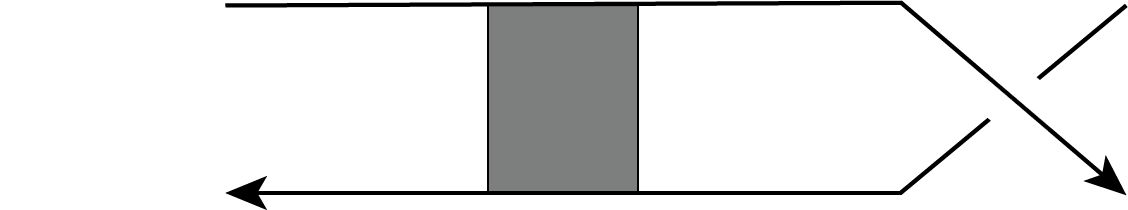}
\caption{Transformation starting with a stuck crossing.}
    \label{TRNAfolding2}
\end{figure}

In the following example we will consider an RNA folding and construct its corresponding stuck link diagram via the transformation defined above. Since the strands in an arc diagram can be closed in several ways, we will avoid any ambiguity by following the approach in \cites{B, TLKL} of self-closure.  

\begin{example}\label{RNAFoldingStuck}
We include arc diagrams of several RNA foldings and their corresponding stuck knot/link diagrams after the transformation $T$ and self-closure has been applied. In the first example, we will consider the following arc diagram of the following RNA folding with one strand and the corresponding stuck knot after applying the transformation $T$ and performing the self-closure. In this case, since the arc diagram only has one strand, the self-closure means connecting the loose ends.

\[
\includegraphics[scale=.4]{grayRNAFolding}
\hspace{1cm}
        \begin{tikzpicture}[use Hobby shortcut,scale=.7]
\begin{knot}[
  consider self intersections = true,
  clip width=5,
  flip crossing/.list={2,6,7,8}
]
\strand[-stealth] ([closed]-1,1)..(1,-1)..(-1,-3)..(0,-4)..(1,-3)..(-1,-1)..(1,1)..(0,2);
\draw[decoration={markings,mark=at position .35 with
  {\arrow[scale=1.5,>=stealth]{>}}},postaction={decorate}]([closed]-1,1)..(1,-1)..(-1,-3)..(0,-4)..(1,-3)..(-1,-1)..(1,1)..(0,2);
\end{knot}
\draw[line width=1.5mm,red] (.6,.4)..(-.59,-.25);
\draw[latex'-latex', thick](-4.5,0)..(-3,0);
\node[above] at (-3.75,0) {\tiny $T$};
\end{tikzpicture}
\]
Next, we consider an arc diagram of an RNA folding with two strands and the corresponding stuck link diagram after we apply the transformation $T$ and perform the self-closure. In this case, since the arc diagram has two strands, the self-closure means that we connect the loose ends of one strand to each other and the loose ends of the other strand to each other.
\[
\includegraphics[scale=.4]{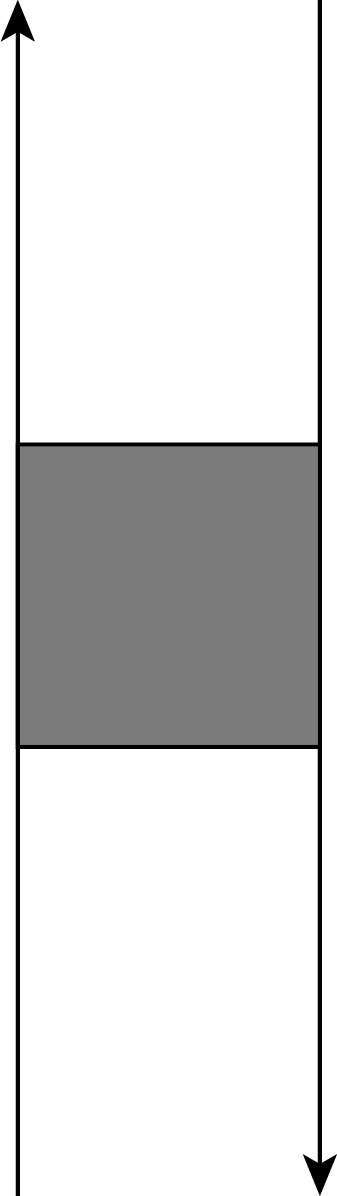}
\hspace{1cm}
\begin{tikzpicture}[use Hobby shortcut, scale=.8]
diagram on the left
\begin{knot}[
clip width=4,
  flip crossing=2
]
\strand(-1,0) circle[radius=2cm];
    \draw[decoration={markings,mark=at position .03 with
    {\arrow[scale=1.5,>=stealth]{<}}},postaction={decorate}](-1,0) circle[radius=2cm];
\strand (1,0) circle[radius=2cm];
    \draw[decoration={markings,mark=at position .2 with
    {\arrow[scale=1.5,>=stealth]{<}}},postaction={decorate}] (1,0) circle[radius=2cm];
\end{knot}
\draw [red,line width=2mm,domain=100:135] plot ({2*cos(\x)+1}, {2*sin(\x)});
\draw[latex'-latex', thick](-6,0)..(-4.5,0);
\node[above] at (-5.25,0) {\tiny $T$};
\end{tikzpicture}
\]
\end{example}

In the following section, we will introduce a set of moves that may be applied to stuck link diagrams to simplify the diagram while still representing the same stuck link. For example, in the first example in Example~\ref{RNAFoldingStuck}, the diagram on the right may be simplified by applying an $\Omega b1$ Reidemeister move.

\section{Generating Set for Oriented Reidemeister Moves for Stuck Knots}\label{GenSet}
This section will introduce a generating set of oriented stuck Reidemeister moves. We note that minimal generating sets for the classical Reidemeister moves have been studied by Polyak, see \cite{P}. Polyak proved the following set is a minimal generating set of the classical Reidemeister moves, see Figure \ref{rmoves}.

\begin{figure}[ht]
    \centering
\begin{tikzpicture}[use Hobby shortcut,scale=.5]
\begin{knot}[
  consider self intersections = true,
  clip width=5,
  flip crossing/.list={5,6,7,8}
]
\strand[-stealth](-5,-3)..(-5,3);
\draw[<->] (-4.5,0)..(-3,0);
\strand[-stealth](-3,-3)..(-2.5,-2)..(-1,.5)..(1,0)..(-1,-.5)..(-2.5,2)..(-3,3); 
\end{knot}
\node[above] at (-3.75,-0) {\tiny $\Omega 1a$};
\end{tikzpicture}
\hspace{1cm}
\begin{tikzpicture}[use Hobby shortcut,scale=.5]
\begin{knot}[
  consider self intersections = true,
  clip width=5,
  flip crossing/.list={5,6,7,8}
]
\strand[stealth-](-5,-3)..(-5,3);
\draw[<->] (-4.5,0)..(-3,0);
\strand[stealth-](-3,-3)..(-2.5,-2)..(-1,.5)..(1,0)..(-1,-.5)..(-2.5,2)..(-3,3); 
\end{knot}
\node[above] at (-3.75,-0) {\tiny $\Omega 1b$};
\end{tikzpicture}

\vspace{1cm}
\begin{tikzpicture}[use Hobby shortcut,scale=.5]
\begin{knot}[
  clip width=5,
  flip crossing/.list={5,6,7,8}
]
\strand[-stealth](-1,-2)..(-2,-1)..(-3,0)..(-2,1)..(-1,2);   
\strand[-stealth](-3,-2)..(-2,-1)..(-1,0)..(-2,1)..(-3,2);
\draw[<->] (.5,0)..(2.5,0);
\strand[-stealth](4,-2)..(4,2);    
\strand[-stealth](6,-2)..(6,2);
\end{knot}
\node[above] at (1.5,0) {\tiny $\Omega 2a$};
\end{tikzpicture}

\vspace{1cm}
\begin{tikzpicture}[use Hobby shortcut,scale=.5]
\begin{knot}[
  consider self intersections = true,
  clip width=5,
  flip crossing/.list={3,6,7,8}
]
\strand[-stealth](-3,-3.5)..(-7,-3.5);
\strand[stealth-](-4.5,-1)..(-6,-5);   
\strand[-stealth](-5.5,-1)..(-4,-5);
 \draw[<->] (-2.5,-3)..(-1.5,-3);
\strand[-stealth](3,-2.5)..(-1,-2.5);
\strand[stealth-](2,-1)..(.5,-5); 
\strand[-stealth](0,-1)..(1.5,-5);
\end{knot}
\node[above] at (-2,-3) {\tiny $\Omega 3a$};
\end{tikzpicture}
    \caption{A minimial generating set for the classical Reidemeister Moves.}
    \label{rmoves}
\end{figure}
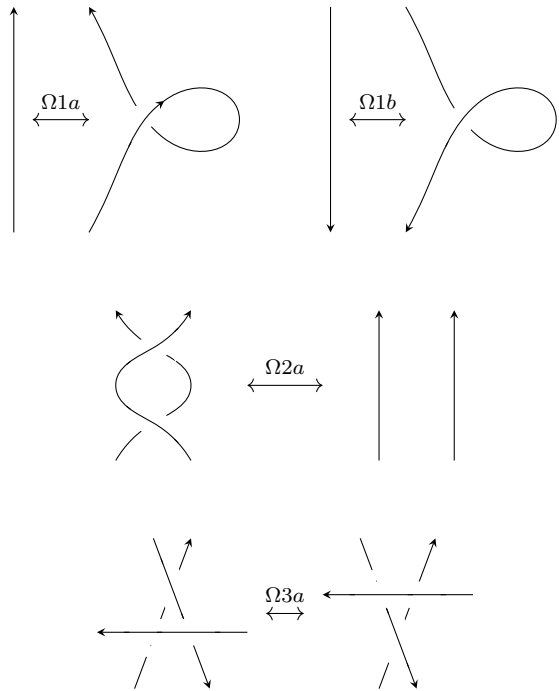

Furthermore, Bataineh, Elhamdadi, Hajij, and Youmans investigated generating sets for Reidemeister-like moves for singular links. They were able to show that in the presence of the Reidemeister moves in Figure \ref{rmoves}, the moves in Figure \ref{singmoves} are enough to generate all generalized Reidemeister moves, see \cite{BEHY}. We will use the naming convention for the oriented classical and singular Reidemeister moves used in \cite{BEHY}.

\begin{figure}[ht]
    \centering

\begin{tikzpicture}[use Hobby shortcut,scale=.5]
\begin{knot}[
  clip width=4,
  flip crossing/.list={5,6,7,8}
]
\strand[-stealth](-1,-2)..(-2,-1)..(-3,0)..(-2,1)..(-1,2);   
\strand[-stealth](-3,-2)..(-2,-1)..(-1,0)..(-2,1)..(-3,2);
\draw[<->] (.5,0)..(2.5,0);
\strand[-stealth](4,-2)..(5,-1)..(6,0)..(5,1)..(4,2);    
\strand[-stealth](6,-2)..(5,-1)..(4,0)..(5,1)..(6,2);
\end{knot}
\node[circle,draw=black, fill=black, inner sep=0pt,minimum size=7pt] (a) at (-2,-1) {};
\node[circle,draw=black, fill=black, inner sep=0pt,minimum size=7pt] (a) at (5,1) {};
\node[above] at (1.5,0) {\tiny $\Omega 5a$};
\end{tikzpicture}

\vspace{1cm}
\begin{tikzpicture}[use Hobby shortcut,scale=.5]
\begin{knot}[
  consider self intersections = true,
  clip width=4,
  flip crossing/.list={3,6,7,8}
]
\strand[-stealth](-3,-3.5)..(-7,-3.5);
\strand[stealth-](-4.5,-1)..(-6,-5);   
\strand[-stealth](-5.5,-1)..(-4,-5);
 \draw[<->] (-2.5,-3)..(-1.5,-3);
\strand[-stealth](3,-2.5)..(-1,-2.5);
\strand[stealth-](2,-1)..(.5,-5); 
\strand[-stealth](0,-1)..(1.5,-5);
\end{knot}
\node[circle,draw=black, fill=black, inner sep=0pt,minimum size=7pt] (a) at (-5,-2.35) {};
\node[circle,draw=black, fill=black, inner sep=0pt,minimum size=7pt] (a) at (1,-3.67) {};
\node[above] at (-2,-3) {\tiny $\Omega 4a$};
\end{tikzpicture}

\vspace{1cm}
\begin{tikzpicture}[use Hobby shortcut,scale=.5]
\begin{knot}[
  consider self intersections = true,
  clip width=4,
  flip crossing/.list={7,8}
]
\strand[stealth-](-4.5,-1)..(-6,-5);   
\strand[-stealth](-5.5,-1)..(-4,-5);
\strand[-stealth](-3,-3.5)..(-7,-3.5);
 \draw[<->] (-2.5,-3)..(-1.5,-3);
\strand[stealth-](2,-1)..(.5,-5); 
\strand[-stealth](0,-1)..(1.5,-5);
\strand[-stealth](3,-2.5)..(-1,-2.5);
\end{knot}
\node[circle,draw=black, fill=black, inner sep=0pt,minimum size=7pt] (a) at (-5,-2.35) {};
\node[circle,draw=black, fill=black, inner sep=0pt,minimum size=7pt] (a) at (1,-3.67) {};
\node[above] at (-2,-3) {\tiny $\Omega 4e$};
\end{tikzpicture}
    \caption{Additional singular Reidemeister moves}
    \label{singmoves}
\end{figure}
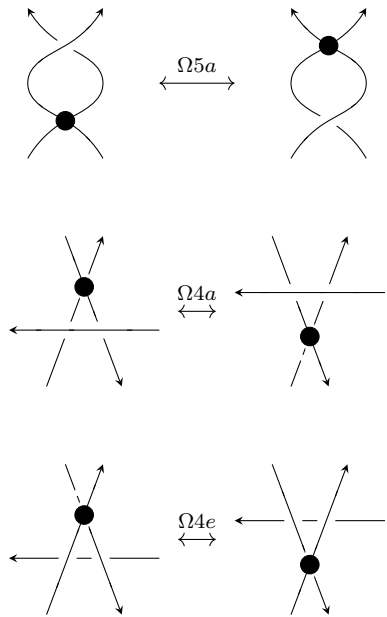

Each singular Reidemeister move will give rise to two stuck Reidemeister moves. The singular Reidemeister move $\Omega 5a$ gives rise to the following two stuck moves, see Figure~\ref{stuck1}.

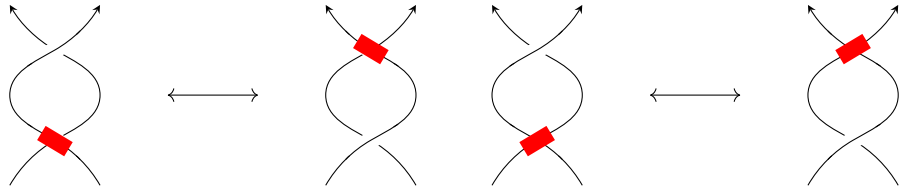
\begin{figure}[ht]
    \centering

\begin{tikzpicture}[use Hobby shortcut,scale=.6]
\begin{knot}[
  clip width=4,
  flip crossing/.list={5,6,7,8}
]
\strand[-stealth](-1,-2)..(-2,-1)..(-3,0)..(-2,1)..(-1,2);   
\strand[-stealth](-3,-2)..(-2,-1)..(-1,0)..(-2,1)..(-3,2);
\draw[<->] (.5,0)..(2.5,0);
\strand[-stealth](4,-2)..(5,-1)..(6,0)..(5,1)..(4,2);    
\strand[-stealth](6,-2)..(5,-1)..(4,0)..(5,1)..(6,2);
\end{knot}
\draw[line width=2.2mm,red] (-2.3,-.84)..(-1.7,-1.2);
\draw[line width=2.2mm,red] (5.3,.84)..(4.7,1.2);
\end{tikzpicture}
\hspace{.8cm}
\begin{tikzpicture}[use Hobby shortcut,scale=.6]
\begin{knot}[
  clip width=4,
  flip crossing/.list={5,6,7,8}
]
\strand[-stealth](-1,-2)..(-2,-1)..(-3,0)..(-2,1)..(-1,2);   
\strand[-stealth](-3,-2)..(-2,-1)..(-1,0)..(-2,1)..(-3,2);
\draw[<->] (.5,0)..(2.5,0);
\strand[-stealth](4,-2)..(5,-1)..(6,0)..(5,1)..(4,2);    
\strand[-stealth](6,-2)..(5,-1)..(4,0)..(5,1)..(6,2);
\end{knot}
\draw[line width=2.2mm,red] (-2.3,-1.2)..(-1.7,-.84);
\draw[line width=2.2mm,red] (5.3,1.2)..(4.7,.84);
\end{tikzpicture}
\caption{Oriented stuck Reidemeister moves derived from $\Omega 5a$.}
\label{stuck1}
\end{figure}

Similarly, the singular Reidemeister move $\Omega 4a$ gives rise to the following two stuck moves, see Figure~\ref{stuck2}.

\begin{figure}[ht]
    \centering
\begin{tikzpicture}[use Hobby shortcut,scale=.6]
\begin{knot}[
  consider self intersections = true,
  clip width=4,
  flip crossing/.list={3,6,7,8}
]
\strand[-stealth](-3,-3.5)..(-7,-3.5);
\strand[stealth-](-4.5,-1)..(-6,-5);   
\strand[-stealth](-5.5,-1)..(-4,-5);
 \draw[<->] (-2.5,-3)..(-1.5,-3);
\strand[-stealth](3,-2.5)..(-1,-2.5);
\strand[stealth-](2,-1)..(.5,-5); 
\strand[-stealth](0,-1)..(1.5,-5);
\end{knot}
\draw[line width=2.2mm,red] (-5.15,-2.7)..(-4.86,-2);
\draw[line width=2.2mm,red] (.9,-4)..(1.16,-3.3);
\end{tikzpicture}
\hspace{.4cm}
\begin{tikzpicture}[use Hobby shortcut,scale=.6]
\begin{knot}[
  consider self intersections = true,
  clip width=4,
  flip crossing/.list={3,6,7,8}
]
\strand[-stealth](-3,-3.5)..(-7,-3.5);
\strand[stealth-](-4.5,-1)..(-6,-5);   
\strand[-stealth](-5.5,-1)..(-4,-5);
 \draw[<->] (-2.5,-3)..(-1.5,-3);
\strand[-stealth](3,-2.5)..(-1,-2.5);
\strand[stealth-](2,-1)..(.5,-5); 
\strand[-stealth](0,-1)..(1.5,-5);
\end{knot}
\draw[line width=2.2mm,red] (-5.15,-2)..(-4.86,-2.7);
\draw[line width=2.2mm,red] (.9,-3.3)..(1.16,-4);
\end{tikzpicture}
\caption{Oriented stuck Reidemeister moves derived from $\Omega 4a$.}
\label{stuck2}
\end{figure}
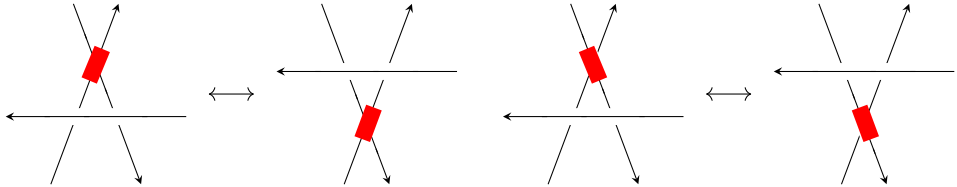
Lastly, we derive the following two stuck Reidemeister moves from the $\Omega 4e$ singular Reidemeister move, see Figure~\ref{stuck3}.

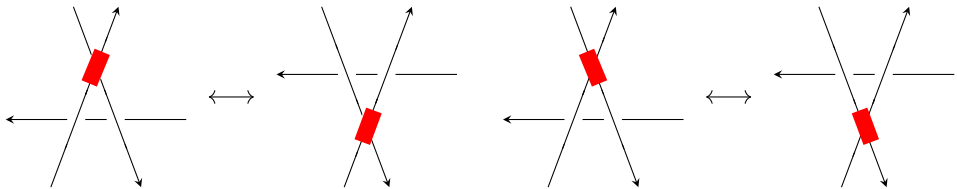
\begin{figure}[ht]
    \centering
    \begin{tikzpicture}[use Hobby shortcut,scale=.6]
\begin{knot}[
  consider self intersections = true,
  clip width=4,
  flip crossing/.list={7,8}
]
\strand[stealth-](-4.5,-1)..(-6,-5);   
\strand[-stealth](-5.5,-1)..(-4,-5);
\strand[-stealth](-3,-3.5)..(-7,-3.5);
 \draw[<->] (-2.5,-3)..(-1.5,-3);
\strand[stealth-](2,-1)..(.5,-5); 
\strand[-stealth](0,-1)..(1.5,-5);
\strand[-stealth](3,-2.5)..(-1,-2.5);
\end{knot}
\draw[line width=2.2mm,red] (-5.15,-2.7)..(-4.86,-2);
\draw[line width=2.2mm,red] (.9,-4)..(1.16,-3.3);
\end{tikzpicture}
\hspace{.4cm}
\begin{tikzpicture}[use Hobby shortcut,scale=.6]
\begin{knot}[
  consider self intersections = true,
  clip width=4,
  flip crossing/.list={7,8}
]
\strand[stealth-](-4.5,-1)..(-6,-5);   
\strand[-stealth](-5.5,-1)..(-4,-5);
\strand[-stealth](-3,-3.5)..(-7,-3.5);
 \draw[<->] (-2.5,-3)..(-1.5,-3);
\strand[stealth-](2,-1)..(.5,-5); 
\strand[-stealth](0,-1)..(1.5,-5);
\strand[-stealth](3,-2.5)..(-1,-2.5);
\end{knot}
\draw[line width=2.2mm,red] (-5.15,-2)..(-4.86,-2.7);
\draw[line width=2.2mm,red] (.9,-3.3)..(1.16,-4);
\end{tikzpicture}
    \caption{Oriented stuck Reidemeister moves derived from $\Omega 4e$}
\label{stuck3}
\end{figure}
Since each singular crossing has two possible ways it can be stuck, each move in Figure~\ref{singmoves} gives rise to two stuck moves. Therefore, in the presence of the oriented classical Reidemeister moves, Figure~\ref{rmoves}, the moves in Figure~\ref{stuck1}, \ref{stuck2} and \ref{stuck3} form a generating set of oriented stuck Reidemeister moves. The proof that this is a generating set for the stuck Reidemeister moves follows directly from the proof of the generating set of the generalized Reidemeister moves in \cite{BEHY}.

\section{Quandles and Singquandles} \label{quandlesreview}
In this section, we will review the definitions and provide basic examples of quandles and singquandles and introduce an algebraic structure that axiomatizes the oriented stuck Reidemeister moves. We will first review the basics of quandle theory; more details on the topic can be found in  \cites{EN, Joyce, Matveev}. We will use the following convention at a classical crossing, see Figure~\ref{crossingrule}.
\begin{figure}[ht]
    \centering
    \begin{tikzpicture}[use Hobby shortcut]
\begin{knot}[
  consider self intersections = true,
  clip width=5,
  flip crossing/.list={1,5,6,7,8}
]
\strand[-stealth] (-6,1)..(-4,-1);
\strand[-stealth] (-4,1)..(-6,-1);
\end{knot}
\node[above] at (-6,1) {\tiny $x$};
\node[above] at (-4,1) {\tiny $y$};
\node[below] at (-4,-1) {\tiny $x*y$};
\end{tikzpicture}
\hspace{1cm}
    \begin{tikzpicture}[use Hobby shortcut]
\begin{knot}[
  consider self intersections = true,
  clip width=5,
  flip crossing/.list={1,5,6,7,8}
]
\strand[-stealth] (-4,1)..(-6,-1);
\strand[-stealth] (-6,1)..(-4,-1);
\end{knot}
\node[above] at (-6,1) {\tiny $y$};
\node[above] at (-4,1) {\tiny $x$};
\node[below] at (-6,-1) {\tiny $x\bar{*}y$};
\end{tikzpicture}
    \caption{Coloring rules.}
    \label{crossingrule}
\end{figure}
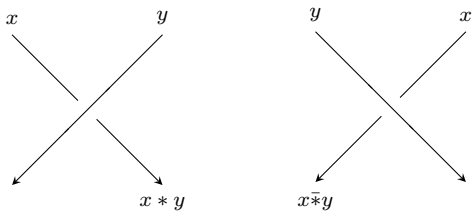

\begin{definition}\label{quandle}  
A set $(X,\ast)$ is called a \emph{quandle} if the following three identities are satisfied.
\begin{eqnarray*}
& &\mbox{\rm (i) \ }   \mbox{\rm  For all $x \in X$,
$x* x =x$.} \label{axiom1} \\
& & \mbox{\rm (ii) \ }\mbox{\rm For all $y,z \in X$, there is a unique $x \in X$ such that 
$ x*y=z$.} \label{axiom2} \\
& &\mbox{\rm (iii) \ }  
\mbox{\rm For all $x,y,z \in X$, we have
$ (x*y)*z=(x*z)*(y*z). $} \label{axiom3} 
\end{eqnarray*}
\end{definition} 

From Axiom (ii) of Definition~\ref{quandle} we can write the element $x$ as $z \bar{*} y = x$. Notice that this operation $\bar{*}$ defines a quandle structure on $X$. A quandle's axioms correspond to the three Reidemeister moves of types I, II, and III. The typical examples are: (1) Alexander quandles given by $x*y=tx+(1-t)y$ on a $\mathbb{Z}[t,t^{-1}]$-module $M$ and (2) conjugation quandle defined on groups $G$ with $x*y=yxy^{-1}$.

Next, we review the basics of singquandles; more details on the topic can be found in \cites{BEHY, CCE1, CCE2,CCEH}.

\begin{definition}\label{SingQdle}
	Let $(X, *)$ be a quandle.  Let $R_1$ and $R_2$ be two maps from $X \times X$ to $X$.  The quadruple $(X, *, R_1, R_2)$ is called an {\it oriented singquandle} if the following axioms are satisfied:
	\begin{eqnarray}
		R_1(x\bar{*}y,z)*y&=&R_1(x,z*y) \;\;\;\text{coming from $\Omega$4a} \label{eq1}\\
		R_2(x\bar{*}y, z) & =&  R_2(x,z*y)\bar{*}y \;\;\;\text{coming from $\Omega$4a}\label{eq2}\\
	      (y\bar{*}R_1(x,z))*x   &=& (y*R_2(x,z))\bar{*}z \;\;\;\text{coming from $\Omega$4e } \label{eq3}\\
R_2(x,y)&=&R_1(y,x*y)  \;\;\;\text{coming from $\Omega$5a} \label{eq4}\\
R_1(x,y)*R_2(x,y)&=&R_2(y,x*y)  \;\;\;\text{coming from $\Omega$5a} \label{eq5}	
\end{eqnarray}	
\end{definition}

Similar to the case of quandles, a singquandle's axioms are motivated by the generalized Reidemeister moves following classical coloring rule in Figure~\ref{crossingrule} and singular coloring rule in Figure~\ref{singrule}.

\begin{figure}[ht]
    \centering
    \begin{tikzpicture}[use Hobby shortcut]
\begin{knot}[
  consider self intersections = true,
  clip width=5,
  flip crossing/.list={1,5,6,7,8}
]
\draw[-stealth] (-6,1)..(-4,-1);
\draw[-stealth] (-4,1)..(-6,-1);
\end{knot}
\node[circle,draw=black, fill=black, inner sep=0pt,minimum size=6pt] (a) at (-5,0) {};
\node[above] at (-6,1) {\tiny $x$};
\node[above] at (-4,1) {\tiny $y$};
\node[below] at (-4,-1) {\tiny $R_1(x,y)$};
\node[below] at (-6,-1) {\tiny $R_2(x,y)$};
\end{tikzpicture}
    \caption{Coloring rule at a singular crossing.}
    \label{singrule}
\end{figure}
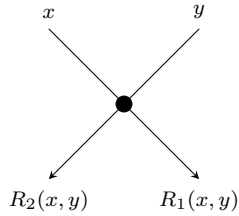

The following two examples were given in \cite{BEHY}.
\begin{example}\label{generalizedaffinesingquandle}
For an invertible element $a \in X=\mathbb{Z}_n$, consider the quandle operation $x*y = ax+(1-a)y$.  Define the maps $R_1$ and $R_2$ by $R_1(x,y) = bx + (1 - b)y$ and $R_2(x,y) = a(1 - b)x + (1 - a(1 - b))y$, then $(X,*,R_1,R_2)$  is an oriented singquandle.
\end{example}

\begin{example}
	Consider the conjugation quandle, $x*y=y^{-1}xy$, on a non-abelian group  $X=G$.  Then $(X, *, R_1, R_2)$ is a singquandle, where for all $x, y \in G$ the maps $R_1$ and $R_2$ are given by:
	\begin{enumerate}
    		\item  $R_1(x,y)=x$ and $R_2(x, y)=y$.
        \item  $R_1(x,y)=xyxy^{-1}x^{-1}$ and $R_2(x, y)=xyx^{-1}$.
        \item  $R_1(x,y)=y^{-1}xy$ and $R_2(x, y)=y^{-1}x^{-1}yxy$.
          \item  $R_1(x,y)=xy^{-1}x^{-1}yx,$ and $R_2(x, y)=x^{-1}y^{-1}xy^2$.
\item  $R_1(x,y)=y(x^{-1}y)^n$ and $R_2(x, y)=(y^{-1}x)^{n+1}y$, where $n \geq 1$.
\end{enumerate}
 \end{example}   

\section{Algebraic Structures for Stuck Knots: Stuquandles} \label{stuckq}
Since the oriented stuck Reidemeister moves are a generalization of the oriented generalized Reidemeister moves, this suggests introducing a new algebraic structure for coloring oriented stuck knots and links with new operations at stuck crossings.

\begin{definition}\label{stuquandle}
	Let $(X, *,R_1,R_2 )$ be a singquandle.  Let $R_3$ and $R_4$ are maps from $X \times X$ to $X$.  If $R_3$ and $R_4$ satisfy the following axioms for all $x,y,z \in X$:
	\begin{eqnarray}
		R_3(y,x)*R_4(y,x)&=&R_4(x*y,y),\label{eq6}\\
	    R_4(y,x)&=&R_3(x*y,y),\label{eq7}\\
	    R_3(y*x,z)& =&R_3(y,z\bar{*}x)*x   ,\label{eq8}\\
		R_4(y,z\bar{*}x)&=&R_4(y*x,z)\bar{*}x, \label{eq9} \\
	    (x*R_4(y,z))\bar{*}y &=& (x\bar{*}R_3(y,z))*z,\label{eq10}
\end{eqnarray}	
then the 6-tuple $(X, *, R_1, R_2, R_3, R_4)$ is called an \emph{oriented stuquandle}.
\end{definition}
\begin{remark}
In this paper we will only consider oriented stuck knots, therefore, we will refer to oriented stuquandles by simply stuquandles.
\end{remark}

\begin{definition}\label{colorrule}
Let $(X, *, R_1,R_2,R_3,R_4)$ be a stuquandle and $L$ an oriented stuck diagram. Then a \emph{coloring of $L$ by $X$} is an assignment of elements of $X$ to the semiarcs of $L$ obeying the following crossing rules:

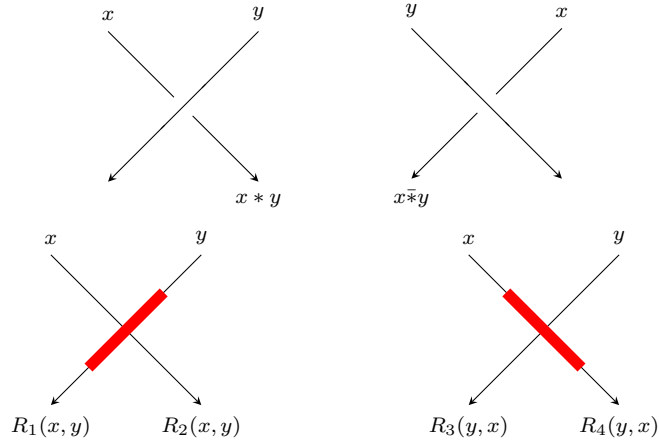
\begin{figure}[ht]
\centering
    \begin{tikzpicture}[use Hobby shortcut]
\begin{knot}[
  consider self intersections = true,
  clip width=5,
  flip crossing/.list={1,5,6,7,8}
]
\strand[-stealth] (-6,1)..(-4,-1);
\strand[-stealth] (-4,1)..(-6,-1);
\end{knot}
\node[above] at (-6,1) {\tiny $x$};
\node[above] at (-4,1) {\tiny $y$};
\node[below] at (-4,-1) {\tiny $x*y$};
\end{tikzpicture}
\hspace{1cm}
    \begin{tikzpicture}[use Hobby shortcut]
\begin{knot}[
  consider self intersections = true,
  clip width=5,
  flip crossing/.list={1,5,6,7,8}
]
\strand[-stealth] (-4,1)..(-6,-1);
\strand[-stealth] (-6,1)..(-4,-1);
\end{knot}
\node[above] at (-6,1) {\tiny $y$};
\node[above] at (-4,1) {\tiny $x$};
\node[below] at (-6,-1) {\tiny $x\bar{*}y$};
\end{tikzpicture}

\begin{tikzpicture}[use Hobby shortcut]
\begin{knot}[
  consider self intersections = true,
  clip width=5,
  flip crossing/.list={1,5,6,7,8}
]
\draw[-stealth] (-3,1)..(-1,-1);
\draw[-stealth] (-1,1)..(-3,-1);
\draw[line width=1.5mm,red] (-2.5,-.5)..(-1.5,.5);
\end{knot}
\node[above] at (-3,1) {\tiny $x$};
\node[above] at (-1,1) {\tiny $y$};
\node[below] at (-3,-1) {\tiny $R_1(x,y)$};
\node[below] at (-1,-1) {\tiny $R_2(x,y)$};
\end{tikzpicture}
\hspace{2cm}
\begin{tikzpicture}[use Hobby shortcut]
\begin{knot}[
  consider self intersections = true,
  clip width=5,
  flip crossing/.list={1,5,6,7,8}
]
\draw[-stealth] (-3,1)..(-1,-1);
\draw[-stealth] (-1,1)..(-3,-1);
\draw[line width=1.5mm,red] (-2.5,.5)..(-1.5,-.5);
\end{knot}
\node[above] at (-3,1) {\tiny $x$};
\node[above] at (-1,1) {\tiny $y$};
\node[below] at (-3,-1) {\tiny $R_3(y,x)$};
\node[below] at (-1,-1) {\tiny $R_4(y,x)$};
\end{tikzpicture}
\caption{Coloring rules at crossings.}
\label{newrule}
\end{figure}
\end{definition}

The stuquandle axioms are motivated by the oriented stuck Reidemeister moves for stuck link diagrams using the coloring rules presented in Figure~\ref{newrule}. Therefore, we obtain the following result. 


\begin{theorem}\label{counting}
Let $L$ be an oriented stuck link diagram. For any finite stuquandle $(X,*,R_1,R_2,R_3,R_4)$, the number of colorings of $L$ by $X$ are invariant under the oriented stuck Reidemeister moves and defines, $Col_X(L)$, the stuquandle counting invariant. 
\end{theorem}

\begin{proof}
As in classical knot theory, there is a one-to-one correspondence between colorings before and after each move in the generating set of stuck Reidemeester moves. Note that $(X,*)$ is a quandle, therefore, the colorings are invaraint under the oriented classical Reidemeister moves. Therefore, we must check the oriented stuck Reidemeistere moves. First we check the oriented stuck Reidemeister moves derived from $\Omega5a$.
\[
\begin{tikzpicture}[use Hobby shortcut,scale=.8]
\begin{knot}[
  clip width=4,
  flip crossing/.list={5,6,7,8}
]
\strand[-stealth](-1,-2)..(-2,-1)..(-3,0)..(-2,1)..(-1,2);   
\strand[-stealth](-3,-2)..(-2,-1)..(-1,0)..(-2,1)..(-3,2);
\draw[<->] (.5,0)..(2.5,0);
\strand[-stealth](4,-2)..(5,-1)..(6,0)..(5,1)..(4,2);    
\strand[-stealth](6,-2)..(5,-1)..(4,0)..(5,1)..(6,2);
\end{knot}
\draw[line width=2.2mm,red] (-2.3,-.84)..(-1.7,-1.2);
\draw[line width=2.2mm,red] (5.3,.84)..(4.7,1.2);
\node[left] at (-3,-2) {\small $y$};
\node[right] at (-1,-2) {\small $x$};
\node[left] at (-3,0) {\small $R_4(y,x)$};
\node[right] at (-1,.5) {\small $R_3(y,x)$};
\node[left] at (-3,2) {\small $R_3(y,x) *R_4(y,x)$};
\node[right] at (-1,2) {\small $R_4(y,x)$};
\node[left] at (4,-2) {\small $y$};
\node[right] at (6,-2) {\small $x$};
\node[left] at (4,0) {\small $x*y$};
\node[right] at (6,0) {\small $y$};
\node[left] at (4,2) {\small $R_4(x*y,y)$};
\node[right] at (6,2) {\small $R_3(x*y,y)$};
\end{tikzpicture}
\]
\[
\begin{tikzpicture}[use Hobby shortcut,scale=.8]
\begin{knot}[
  clip width=4,
  flip crossing/.list={5,6,7,8}
]
\strand[-stealth](-1,-2)..(-2,-1)..(-3,0)..(-2,1)..(-1,2);   
\strand[-stealth](-3,-2)..(-2,-1)..(-1,0)..(-2,1)..(-3,2);
\draw[<->] (.5,0)..(2.5,0);
\strand[-stealth](4,-2)..(5,-1)..(6,0)..(5,1)..(4,2);    
\strand[-stealth](6,-2)..(5,-1)..(4,0)..(5,1)..(6,2);
\end{knot}
\draw[line width=2.2mm,red] (-2.3,-1.2)..(-1.7,-.84);
\draw[line width=2.2mm,red] (5.3,1.2)..(4.7,.84);
\node[left] at (-3,-2) {\small $y$};
\node[right] at (-1,-2) {\small $x$};
\node[left] at (-3,0) {\small $R_2(x,y)$};
\node[right] at (-1,.5) {\small $R_1(x,y)$};
\node[left] at (-3,2) {\small $R_1(x,y) *R_2(x,y)$};
\node[right] at (-1,2) {\small $R_2(x,y)$};
\node[left] at (4,-2) {\small $y$};
\node[right] at (6,-2) {\small $x$};
\node[left] at (4,0) {\small $x*y$};
\node[right] at (6,0) {\small $y$};
\node[left] at (4,2) {\small $R_2(y,x*y)$};
\node[right] at (6,2) {\small $R_1(y,x*y)$};
\end{tikzpicture}
\]
We have axioms (4), (5) in Definition~\ref{SingQdle} and (6), (7) in Definition~\ref{stuquandle}. Next we check the oriented stuck Reidemeister moves derived from $\Omega4a$.
\[
\begin{tikzpicture}[use Hobby shortcut,scale=.8]
\begin{knot}[
  consider self intersections = true,
  clip width=4,
  flip crossing/.list={3,6,7,8}
]
\strand[-stealth](-3,-3.5)..(-7,-3.5);
\strand[stealth-](-4.5,-1)..(-6,-5);   
\strand[-stealth](-5.5,-1)..(-4,-5);
 \draw[<->] (-2.5,-3)..(-1.5,-3);
\strand[-stealth](3,-2.5)..(-1,-2.5);
\strand[stealth-](2,-1)..(.5,-5); 
\strand[-stealth](0,-1)..(1.5,-5);
\end{knot}
\draw[line width=2.2mm,red] (-5.15,-2.7)..(-4.86,-2);
\draw[line width=2.2mm,red] (.9,-4)..(1.16,-3.3);
\node[below] at (-3,-3.5) {\small $x$};
\node[left] at (-5.5,-1) {\small $y$};
\node[left] at (-6,-5) {\small $z$};
\node[right] at (-4.5,-1) {\small $R_4(y,z\bar{*}x)$};
\node[below] at (-4,-5) {\small $R_3(y,z\bar{*}x)*x$};
\node[below] at (3,-2.5) {\small $x$};
\node[left] at (0,-1) {\small $y$};
\node[left] at (.5,-5) {\small $z$};
\node[right] at (2,-1) {\small $R_4(y*x,z)\bar{*}x$};
\node[below] at (2.5,-5) {\small $R_3(y*x,z)$};
\end{tikzpicture}
\]
\[
\begin{tikzpicture}[use Hobby shortcut,scale=.8]
\begin{knot}[
  consider self intersections = true,
  clip width=4,
  flip crossing/.list={3,6,7,8}
]
\strand[-stealth](-3,-3.5)..(-7,-3.5);
\strand[stealth-](-4.5,-1)..(-6,-5);   
\strand[-stealth](-5.5,-1)..(-4,-5);
 \draw[<->] (-2.5,-3)..(-1.5,-3);
\strand[-stealth](3,-2.5)..(-1,-2.5);
\strand[stealth-](2,-1)..(.5,-5); 
\strand[-stealth](0,-1)..(1.5,-5);
\end{knot}
\draw[line width=2.2mm,red] (-5.15,-2)..(-4.86,-2.7);
\draw[line width=2.2mm,red] (.9,-3.3)..(1.16,-4);
\node[below] at (-3,-3.5) {\small $x$};
\node[left] at (-5.5,-1) {\small $y$};
\node[left] at (-6,-5) {\small $z$};
\node[right] at (-4.5,-1) {\small $R_2(z\bar{*}x,y)$};
\node[below] at (-4,-5) {\small $R_1(z \bar{*}x,y)*x$};
\node[below] at (3,-2.5) {\small $x$};
\node[left] at (0,-1) {\small $y$};
\node[left] at (.5,-5) {\small $z$};
\node[right] at (2,-1) {\small $R_2(z,y*x)\bar{*}x$};
\node[below] at (2.5,-5) {\small $R_1(z,y*x)$};
\end{tikzpicture}
\]
We have axioms (2), (3) in Definition~\ref{SingQdle} and (8), (9) in Definition~\ref{stuquandle}. Next we check the oriented stuck Reidemeister moves derived from $\Omega4e$.
\[
\begin{tikzpicture}[use Hobby shortcut,scale=.8]
\begin{knot}[
  consider self intersections = true,
  clip width=4,
  flip crossing/.list={7,8}
]
\strand[stealth-](-4.5,-1)..(-6,-5);   
\strand[-stealth](-5.5,-1)..(-4,-5);
\strand[-stealth](-3,-3.5)..(-7,-3.5);
 \draw[<->] (-2.5,-3)..(-1.5,-3);
\strand[stealth-](2,-1)..(.5,-5); 
\strand[-stealth](0,-1)..(1.5,-5);
\strand[-stealth](3,-2.5)..(-1,-2.5);
\end{knot}
\draw[line width=2.2mm,red] (-5.15,-2.7)..(-4.86,-2);
\draw[line width=2.2mm,red] (.9,-4)..(1.16,-3.3);
\node[below] at (-3,-3.5) {\small $x$};
\node[below] at (-7.5,-3.5) {\small $(x\bar{*}R_3(y,z))*z$};
\node[left] at (-5.5,-1) {\small $y$};
\node[left] at (-6,-5) {\small $z$};
\node[right] at (-4.5,-1) {\small $R_4(y,z)$};
\node[below] at (-4,-5) {\small $R_3(y,z)$};
\node[below] at (3,-2.5) {\small $x$};
\node[above] at (-1.5,-2.5) {\small $(x*R_4(y,z))\bar{*}y$};
\node[left] at (0,-1) {\small $y$};
\node[left] at (.5,-5) {\small $z$};
\node[right] at (2,-1) {\small $R_4(y,z)$};
\node[below] at (2.5,-5) {\small $R_3(y,z)$};
\end{tikzpicture}
\]
\[
\begin{tikzpicture}[use Hobby shortcut,scale=.8]
\begin{knot}[
  consider self intersections = true,
  clip width=4,
  flip crossing/.list={7,8}
]
\strand[stealth-](-4.5,-1)..(-6,-5);   
\strand[-stealth](-5.5,-1)..(-4,-5);
\strand[-stealth](-3,-3.5)..(-7,-3.5);
 \draw[<->] (-2.5,-3)..(-1.5,-3);
\strand[stealth-](2,-1)..(.5,-5); 
\strand[-stealth](0,-1)..(1.5,-5);
\strand[-stealth](3,-2.5)..(-1,-2.5);
\end{knot}
\draw[line width=2.2mm,red] (-5.15,-2)..(-4.86,-2.7);
\draw[line width=2.2mm,red] (.9,-3.3)..(1.16,-4);
\node[below] at (-3,-3.5) {\small $x$};
\node[below] at (-7.5,-3.5) {\small $(x\bar{*}R_1(z,y))*z$};
\node[left] at (-5.5,-1) {\small $y$};
\node[left] at (-6,-5) {\small $z$};
\node[right] at (-4.5,-1) {\small $R_2(z,y)$};
\node[below] at (-4,-5) {\small $R_1(z,y)$};
\node[below] at (3,-2.5) {\small $x$};
\node[above] at (-1.5,-2.5) {\small $(x*R_2(z,y))\bar{*}y$};
\node[left] at (0,-1) {\small $y$};
\node[left] at (.5,-5) {\small $z$};
\node[right] at (2,-1) {\small $R_2(z,y)$};
\node[below] at (2.5,-5) {\small $R_1(z,y)$};
\end{tikzpicture}
\]
We have axioms (1) in Definition~\ref{SingQdle} and (10) in Definition~\ref{stuquandle}.
\end{proof}

Next, we provide two examples of stuquandles that will be useful in distinguishing oriented stuck knots and links.

\begin{example}\label{generalizedaffinestuquandle}
Let $X=\mathbb{Z}_n$ with the quandle operation $x*y = ax+(1-a)y$, where $a$ is invertible so that $x\ \bar{*}\ y = a^{-1}x+(1-a^{-1})y$.  Now let $R_1(x,y) = bx+cy$, then by axiom~(\ref{eq4}) of Definition~\ref{SingQdle} we have $R_2(x,y) = acx + [c(1-a) + b]y$. By substituting these expressions into axiom~(\ref{eq1}) of Definition~\ref{SingQdle} we find the relation $c= 1 - b$. Substituting, we find that the following is an oriented singquandle for any invertible $a$ and any $b$ in $\mathbb{Z}_n$:
	\begin{eqnarray}
    	x*y &=& ax + (1-a)y, \label{sing1} \\
        R_1(x,y) &=& bx + (1 - b)y, \label{sing2} \\
        R_2(x,y) &=& a(1 - b)x + (1 - a(1 - b))y. \label{sing3}
    \end{eqnarray}
    Next, let $R_3(x,y) = dx+ey$, then by axiom~(\ref{eq7}) of Definition~\ref{stuquandle}, we have $R_4(x,y) =  [d(1-a) + e]x+ady$.  By substituting these expressions into axiom~(\ref{eq6}) of Definition~\ref{stuquandle} we obtain no constraints on the coefficients $d$ and $e$.
    Axiom~(\ref{eq8}) of Definition~\ref{stuquandle} imposes that $d=1-e$, while axioms~(\ref{eq9}, \ref{eq10}) of Definition~\ref{stuquandle} introduce no constraint. 
    Substituting, we obtain
\begin{eqnarray}
  R_3(x,y)&=&(1-e)x+ey,\\
  R_4(x,y)&=&(1-a(1-e))x+a(1-e)y.
\end{eqnarray}
 Therefore, $(\mathbb{Z}_n,*,R_1,R_2,R_3,R_4)$ is an oriented stuquandle where $a$ is invertible element of $\mathbb{Z}_n$ and any $b,e \in \mathbb{Z}_n$.

\end{example}



\begin{example}
Let $\Lambda=\mathbb{Z}[t^{\pm 1},v]$  and let $X$ be a $\Lambda$-module. Let
\begin{eqnarray*}
x\ast y &=& tx+(1-t)y,\\
R_1(x,y) &=& \alpha(a,b,c)x+(1-\alpha(a,b,c))y,\\
R_2(x,y) &=& t[1  - \alpha(a,b,c)] x + [1 -t(1-  \alpha(a,b,c))] y,\\
R_3(x,y) &=& (1-\alpha(d,e,f))x+\alpha(d,e,f)y,\\
R_2(x,y) &=&  [1 -t(1-  \alpha(d,e,f))] x + t[1  - \alpha(d,e,f)] y ,\\
\end{eqnarray*}
where $\alpha(a,b,c)=at+bv+ctv$ and $\alpha(d,e,f)=dt+fv+etv$. Then $(X,*,R_1,R_2,R_3,R_4)$ is an oriented stuquandle, which we call an \textit{Alexander oriented stuquandle}. The fact that $(X,*,R_1,R_2,R_3,R_4)$ is an oriented stuquandle follows from Example \ref{generalizedaffinestuquandle} by straightforward substitution. 
\end{example}

Given a finite set $X=\mathbb{Z}_n$, we can specify a stuquandle
structure on $X$ by explicitly listing the operation tables of the five
stuquandle operations. In practice, it is convenient to put the operation tables together 
into an $n\times 5n$ block matrix, see Example~\ref{block}.

\begin{example}\label{block}
Consider the stuquandle $X=\mathbb{Z}_3$ with $x \ast y=2x+2y$, $R_1(x,y)=x$, $R_2(x,y)=y$, $R_3(x,y)=x$ and $R_4(x,y)=2x+2y$. This is a stuquandle by Example~\ref{generalizedaffinestuquandle} with $a=2$, $b=1$, and $e=0$. We can also specify the stuquandle with the following operation table,
\[
\begin{array}{r|rrr} \ast & 0 & 1 & 2 \\ \hline 0 & 0 & 2 & 1 \\ 1 & 2 & 1 & 0 \\ 2 & 1 & 0 & 2\end{array}\ \ 
\begin{array}{r|rrr} R_1 & 0 & 1 & 2 \\ \hline 0 & 0 & 1 & 2 \\ 1 & 0 & 1 & 2 \\ 2 & 0 & 1 & 2\end{array}\ \ 
\begin{array}{r|rrr} R_2 & 0 & 1 & 2 \\ \hline 0 & 0 & 0 & 0 \\ 1 & 1 & 1 & 1 \\ 2& 2 & 2 & 2\end{array}\ \ 
\begin{array}{r|rrr} R_3 & 0 & 1 & 2 \\ \hline 0 & 0 & 1 & 2 \\ 1 & 0 & 1 & 2 \\ 2 & 0 & 1 & 2\end{array}\ \ 
\begin{array}{r|rrr} R_4 & 0 & 1 & 2 \\ \hline 0 & 0 & 2 & 1 \\ 1 & 2 & 1 & 0 \\ 2 & 1 & 0 & 2\end{array}
\]
and the operation tables can be encoded as the block matrix
\[
\left[\begin{array}{rrr|rrr|rrr|rrr|rrr}
0 & 2 & 1 &  0 & 1 & 2 &  0 & 0 & 0 & 0 & 1 & 2 &  0 & 2 & 1 \\\
2 & 1 & 0 &  0 & 1 & 2 &  1 & 1 & 1 &  0 & 1 & 2 &  2 & 1 & 0\\
1 & 0 & 2 &  0 & 1 & 2 &   2 & 2 & 2 & 0 & 1 & 2 &  1 & 0 & 2
\end{array}\right].
\]
\end{example}


\begin{definition}
Let $D$ be a stuck link diagram of a stuck link $L$ and let $S=\{ a_1,a_2,\dots, a_m\}$ be the set of labels of the arcs in $D$ at classical crossings and semiarcs in $D$ at stuck crossings. We define the \emph{fundamental stuquandle} of $D$ by proceeding as in classical knot theory with the fundamental quandle.

\begin{enumerate}
    \item The set of stuquandle words, $W(S)$, is recursively defined. 
    \begin{enumerate}
        \item $S \subset W(S)$,
        \item If $a_i, a_j \in W(S)$, then 
    \[ a_i*a_j, a_i \bar{*} a_j, R_1(a_i,a_j), R_2(a_i,a_j), R_3(a_i,a_j), R_4(a_i,a_j) \in W(S). \]
    \end{enumerate}
    \item The set $Y$ is the set of \emph{free stuquandle words} which are equivalent classes of $W(S)$ determined by the conditions in Definition~\ref{stuquandle}.
    \item Let $c_1,\dots, c_n$ be the crossings of $D$. Each crossing $c_i$ in $D$ determines a relation $r_i$ on the elements of $Y$.
    \item The \emph{fundamental stuquandle} of $D$, $\mathcal{STQ}(D)$, is the set of equivalence class of words in $W(S)$ determined by the stuquandle conditions and the relations given by the crossings of $D$.
\end{enumerate}
\end{definition}

\begin{theorem}
The isomorphism class $\mathcal{STQ}(L)$ of $\mathcal{STQ}(D)$ is an invariant of oriented stuck links.
\end{theorem}
\begin{proof}
By construction, oriented stuck Reidemeister moves on diagrams induce Tietze transformations on presentations of $\mathcal{STQ}(D)$.
\end{proof}

\begin{example}
In this example, we compute the fundamental stuquandle of the following oriented stuck link.  Consider the stuck Hopf link with one stuck crossing and one positive classical crossing, $L$, with stuck diagram $D$, see Figure~\ref{hopf}. 
\begin{figure}[ht]
\centering
\begin{tikzpicture}[use Hobby shortcut]
diagram on the left
\begin{knot}[
clip width=4,
  flip crossing=4
]
\strand(-1,0) circle[radius=2cm];
    \draw[decoration={markings,mark=at position .03 with
    {\arrow[scale=3,>=stealth]{<}}},postaction={decorate}](-1,0) circle[radius=2cm];
\strand (1,0) circle[radius=2cm];
    \draw[decoration={markings,mark=at position .5 with
    {\arrow[scale=3,>=stealth]{>}}},postaction={decorate}] (1,0) circle[radius=2cm];
\end{knot}
\draw [red,line width=2mm,domain=40:75] plot ({2*cos(\x)-1}, {2*sin(\x)});
\node[left] at (-1,0) {\tiny $z$};
\node[left] at (-2,2) {\tiny $x$};
\node[right] at (2,2) {\tiny $y$};
\end{tikzpicture}
\vspace{.2in}
		\caption{Diagram $D$ of $L$.}
		\label{hopf}
\end{figure}
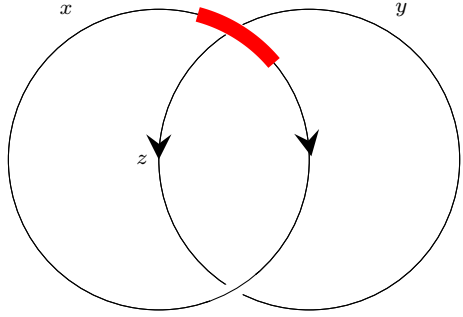
We will label the arcs of the diagram $D$ by $x,y,$ and $z$.  Then the fundamental stuquandle of $L$ is given by ,

\begin{eqnarray*}
\mathcal{STQ}(L)&=&\langle x,y,z \, \vert \, \; x=R_4(y,x), \; z=R_3(y,x), \;z*x=y\rangle\\
&=&\langle x,y\, \vert \,  R_3(y,x)*R_4(y,x)=y \rangle.
\end{eqnarray*}
\end{example}
\begin{definition}
Let $(X,*,R_1,R_2,R_3,R_4)$ and $(Y,\triangleright,S_1,S_2,S_3,S_4)$ be two stuquandles. A map $f:X\longrightarrow Y$ is a \textit{stuquandle homomorphism} if the following conditions are satisfied:
\begin{enumerate}
\item $f(x*y)=f(x)\triangleright f(y)$,
\item $f(R_1(x,y))=S_1(f(x),f(y))$,
\item $f(R_2(x,y))=S_2(f(x),f(y))$,
\item $f(R_3(x,y))=S_3(f(x),f(y))$,
\item $f(R_4(x,y))=S_4(f(x),f(y))$.
\end{enumerate}
If $f$ a bijective stuquandle homomorphism then it is called a \textit{stuquandle isomorphism}. 
\end{definition}

For any oriented stuck link $L$ with diagram $D$, there is an associated fundamental stuquandle, $(\mathcal{STQ}(L), *, R_1, R_2,R_3,R_4)$, and for a finite stuquandle, $(X, \triangleright, R_1', R_2',R_3',R_4')$, the set of stuquandle homomorphisms,
\begin{equation*}
\textup{Hom}(\mathcal{STQ}(L),X) = \left\{ f: \mathcal{STQ}(L) \rightarrow X  \, \middle \vert \, \begin{array}{l}
f( x*y) = f(x) \triangleright f(y),\\ 
f(R_1(x,y))= R_1'(f(x),f(y)),\\
f(R_2(x,y))= R_2'(f(x),f(y)),\\
f(R_3(y,x))= R_3'(f(y),f(x)),\\
f(R_4(y,x))= R_4'(f(y),f(x))
  \end{array}
\right\},     
\end{equation*}
can be used to construct computable invariants for stuck links. For example, by computing the cardinality of this set, we obtain the stuquandle counting invariant defined in Theorem~\ref{counting}. To see this, the assignment of elements of $X$ to the semiarcs in $D$ at a stuck crossing and to the arcs at the classical crossings in $D$ defines a homomorphism $f: \mathcal{STQ}(D) \rightarrow X$ if and only if the coloring conditions in Definition~\ref{colorrule} are satisfied at every stuck and classical crossing in $D$. We can compute the set, $\textup{Hom}(\mathcal{STQ}(L),X) $, by computing the colorings of $D$ by $X$. Thus, we have
\[ Col_X(L) = \vert \textup{Hom}(\mathcal{STQ}(L),X) \vert. \]

\begin{proposition}
Let $L$ be a stuck link. Let $f:(X,*,R_1,R_2,R_3,R_4)\longrightarrow (Y,*,S_1,S_2,S_3,S_4)$ be an isomorphism of stuquandles then there is a one-to-one correspondence between the spaces of colorings $\textup{Hom}(\mathcal{STQ}(L),X)$ and $\textup{Hom}(\mathcal{STQ}(L),Y)$.
\end{proposition}

In the following section, we will consider stuck knots and links derived from singular knots and links known as 2-bouquet graphs of type $K$ and type $L$. We will follow the notation for 2-bouquet graphs found in \cite{O}. We will modify the notation of a 2-bouquet graph by adding $+$ or $-$ in the superscript to denote the derived oriented stuck link with a positive or negative stuck crossing. For an example see Figure~\ref{singstuck}.
\begin{figure}[ht]
\centering
\begin{tikzpicture}[use Hobby shortcut, scale=.6]
\begin{knot}[
clip width=4,
  flip crossing=4
]
\strand(-1,0) circle[radius=2cm];
    \draw[decoration={markings,mark=at position .03 with
    {\arrow[scale=3,>=stealth]{<}}},postaction={decorate}](-1,0) circle[radius=2cm];
\strand (1,0) circle[radius=2cm];
    \draw[decoration={markings,mark=at position .5 with
    {\arrow[scale=3,>=stealth]{>}}},postaction={decorate}] (1,0) circle[radius=2cm];
\end{knot}
\node at (0,1.7) [circle,fill,inner sep=4pt]{};
\end{tikzpicture}
\hspace{.5cm}
\begin{tikzpicture}[use Hobby shortcut,scale=.6]
\begin{knot}[
clip width=4,
  flip crossing=4
]
\strand(-1,0) circle[radius=2cm];
    \draw[decoration={markings,mark=at position .03 with
    {\arrow[scale=3,>=stealth]{<}}},postaction={decorate}](-1,0) circle[radius=2cm];
\strand (1,0) circle[radius=2cm];
    \draw[decoration={markings,mark=at position .5 with
    {\arrow[scale=3,>=stealth]{>}}},postaction={decorate}] (1,0) circle[radius=2cm];
\end{knot}
\draw [red,line width=2mm,domain=105:140] plot ({2*cos(\x)+1}, {2*sin(\x)});
\end{tikzpicture}
\hspace{.5cm}
\begin{tikzpicture}[use Hobby shortcut, scale=.6]
\begin{knot}[
clip width=4,
  flip crossing=4
]
\strand(-1,0) circle[radius=2cm];
    \draw[decoration={markings,mark=at position .03 with
    {\arrow[scale=3,>=stealth]{<}}},postaction={decorate}](-1,0) circle[radius=2cm];
\strand (1,0) circle[radius=2cm];
    \draw[decoration={markings,mark=at position .5 with
   {\arrow[scale=3,>=stealth]{>}}},postaction={decorate}] (1,0) circle[radius=2cm];
\end{knot}
\draw [red,line width=2mm,domain=40:75] plot ({2*cos(\x)-1}, {2*sin(\x)});
\end{tikzpicture}
\vspace{.2in}
		\caption{Oriented singular link $0^l_1$ (left), derived  oriented stuck links $0^{l+}_1$ (middle), and $0^{l-}_1$ (right).}
		\label{singstuck}
\end{figure}
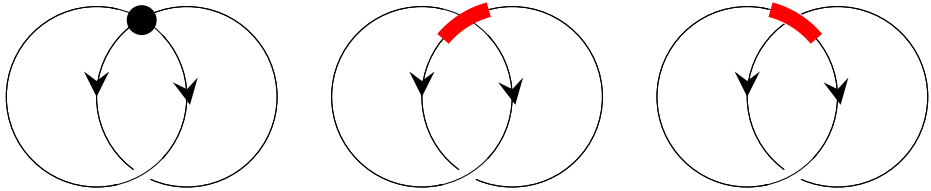

\section{Computations}\label{computations}

\begin{example}
Let $X=(\mathbb{Z}_3,*,R_1,R_2,R_3,R_4)$ be the stuquandle with operation matrix below, 

\[
\left[
\begin{array}{ccc|ccc|ccc|ccc|ccc}
 0& 0& 0    &    0& 1& 1   &   0& 2& 2   &   1& 1& 1 &    1& 1& 2   \\
 1& 1& 1    &    2& 1& 2   &   1& 1& 0   &   1& 2& 2 &    1& 2& 1   \\
 2& 2& 2    &    2& 0& 1   &   1& 2& 1   &   2& 1& 0 &    1& 2& 0  
\end{array}
\right].
\]
The link below 

\[
\begin{tikzpicture}[use Hobby shortcut]
diagram on the left
\begin{knot}[
clip width=4,
  flip crossing=2
]
\strand(-1,0) circle[radius=2cm];
    \draw[decoration={markings,mark=at position .03 with
    {\arrow[scale=3,>=stealth]{<}}},postaction={decorate}](-1,0) circle[radius=2cm];
\strand (1,0) circle[radius=2cm];
    \draw[decoration={markings,mark=at position .5 with
    {\arrow[scale=3,>=stealth]{<}}},postaction={decorate}] (1,0) circle[radius=2cm];
\end{knot}
\draw [red,line width=2mm,domain=105:135] plot ({2*cos(\x)+1}, {2*sin(\x)});
\node[left] at (-1.2,0) {\tiny $x$};
\node[right] at (1,0) {\tiny $R_3(y,x)$};
\node[left] at (-2,2) {\tiny $y$};
\node[right] at (2,2) {\tiny $R_4(y,x)$};
\end{tikzpicture}
\]
has the following coloring equations,
\begin{eqnarray*}
      x &=& R_4(y,x),\\
    y &=& R_3(y,x) *x.       
\end{eqnarray*}
Using the stuquandle defined above the system of coloring equations for $S_1$ has zero solutions. Therefore, $Col_X(S_1)=0$.
The following link 

\[
\begin{tikzpicture}[use Hobby shortcut]
diagram on the left
\begin{knot}[
clip width=4,
  flip crossing=2
]
\strand(-1,0) circle[radius=2cm];
    \draw[decoration={markings,mark=at position .03 with
    {\arrow[scale=3,>=stealth]{<}}},postaction={decorate}](-1,0) circle[radius=2cm];
\strand (1,0) circle[radius=2cm];
    \draw[decoration={markings,mark=at position .5 with
    {\arrow[scale=3,>=stealth]{<}}},postaction={decorate}] (1,0) circle[radius=2cm];
\end{knot}
\draw [red,line width=2mm,domain=45:75] plot ({2*cos(\x)-1}, {2*sin(\x)});
\node[left] at (-1.2,0) {\tiny $x$};
\node[right] at (1,0) {\tiny $R_1(x,y)$};
\node[left] at (-2,2) {\tiny $y$};
\node[right] at (2,2) {\tiny $R_2(x,y)$};
\end{tikzpicture}
\]

has the following coloring equations
\begin{eqnarray*}
      x &=& R_2(x,y),\\
    y &=& R_1(x,y) *x.       
\end{eqnarray*}
Using the stuquandle defined above the system of coloring equations for $S_2$ has two solutions. Therefore, $Col_X(S_2)=2$.
\end{example}

\begin{remark}
Notice from this example that in the case it is possible to obtain \emph{zero} colorings of a stuck knot by a stuquandle. This is impossible when considering the colorings of classical knots by quandles.
\end{remark}

In the following example we get a sense of the effectiveness of the invariant by selecting a specific stuquandle and we compute the stuquandle counting invariant for several stuck knots. Specifically, we will compare our results of the stuquandle counting invariant to the signed sticking number.
\begin{example}\label{ex7.3}
Let $X=\mathbb{Z}_{12}$ be the stuquandle with operation defined by $x \ast y = 11x+2y$ and maps $R_1(x,y) = 10x+3y$, $R_2(x,y)=9x+4y$, $R_3(x,y)=2x+11y$, and $R_4(x,y)=3x+10y$. By Example~\ref{generalizedaffinestuquandle} with $a=11$, $b=10$ and $e=11$ we obtain that $(X,*,R_1,R_2,R_3,R_4)$ is an oriented stuquandle . We compute the stuquandle counting invariant for 8 stuck knots derived from the first 4 singular knots in \cite{O}.

\begin{center}
\begin{tabular}{ c|l |l }
Signed Sticking \# & $Col_X(K)$ & $K$ \\
\hline
-1  & 24 & $0^{k-}_1$, $4^{k-}_1$ \\
    & 48 & $3^{k-}_1$\\
    & 144& $2^{k-}_1$\\
\hline    
1   & 12 &$2^{k+}_1$, $3^{k+}_1$, $4^{k+}_1$ \\
    & 36 & $0^{k+}_1$\\
\end{tabular}
\end{center}

\end{example}

\section{Invariant of Arc Diagrams}\label{app}

We will now compute the stuquandle counting invariant to investigate the topology of RNA structure. We will first consider an arc diagram of RNA folding and then use the transformation defined above to turn the arc diagram into a stuck arc diagram and then close the diagram using the self-closure to get a stuck knot diagram. We will then compute the stuquandle counting invariant using the stuck knot diagram corresponding to the arc diagram of RNA folding.

\begin{example}\label{ex8.1}
Let $X=\mathbb{Z}_{4}$ be the stuquandle with operation defined by $x \ast y = x$ and maps $R_1(x,y) =2x+3y$, $R_2(x,y)=3x+2y$, $R_3(x,y)=y$, and $R_4(x,y)=x$. By Example~\ref{generalizedaffinestuquandle} with $a=1$, $b=2$ and $e=1$ we obtain that $(X,*,R_1,R_2,R_3,R_4)$ is an oriented stuquandle. Consider the following arc diagrams of RNA foldings.


\[
\includegraphics[scale=.5]{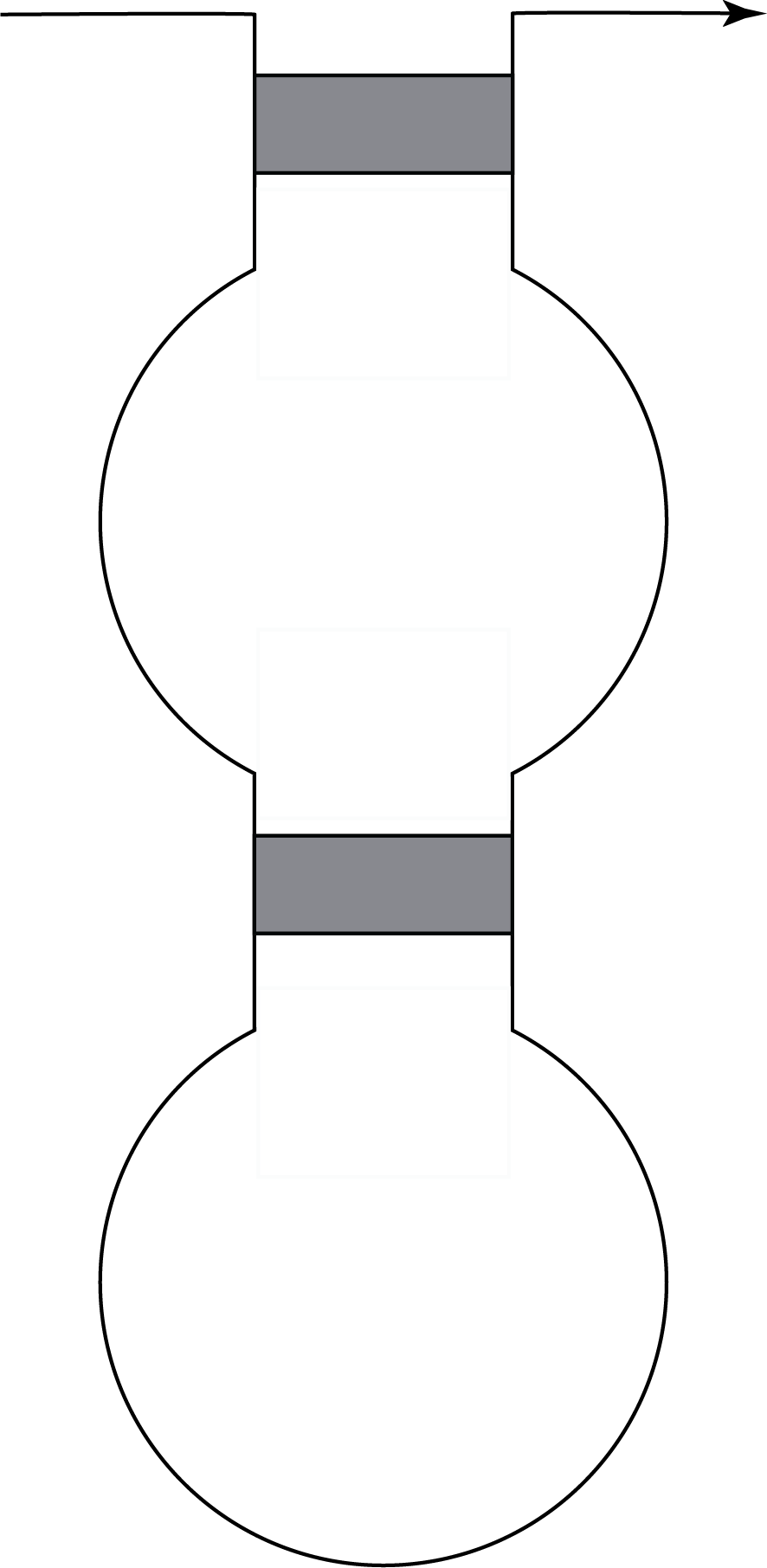}
\hspace{1cm}
\centering
        \begin{tikzpicture}[use Hobby shortcut]
\begin{knot}[
  consider self intersections = true,
  clip width=5,
  flip crossing/.list={2}
]
\strand[-stealth] ([closed]0,5)..(-1,4)..(1,2)..(-1,0)..(0,-1)..(1,0)..(-1,2)..(1,4);
\draw[decoration={markings,mark=at position .1 with
    {\arrow[scale=3,>=stealth]{>}}},postaction={decorate}]([closed]0,5)..(-1,4)..(1,2)..(-1,0)..(0,-1)..(1,0)..(-1,2)..(1,4);
\end{knot}
\draw[line width=1.5mm,red] (-.5,.69)..(.5,1.2);
\draw[line width=1.5mm,red] (-.5,2.78)..(.5,3.35);
\draw[latex'-latex', thick](-4.5,2)..(-3,2);
\node[above] at (-3.75,2) {\tiny $T$};
\node[right] at (1,0) {\small $c$};
\node[left] at (-1,4.5) {\small $b$};
\node[left] at (-1,2) {\small $a$};
\end{tikzpicture}
\]
Using the corresponding stuck knot diagram we obtain the following coloring equations,
\begin{eqnarray*}
a &=& R_3(c,R_3(b,a)),\\
b &=& R_4(b,a),\\ 
c &=& R_4(c,R_3(b,a)).
\end{eqnarray*}
From the first equation we obtain $a=a$, from the second equation we obtain $b=b$, and from the last equation we obtain $c=c$. Therefore, $Col_X \begin{pmatrix}
\includegraphics[scale=.15]{RNAFolding2.png}
\end{pmatrix}  = 64$.
\[
\includegraphics[scale=.5]{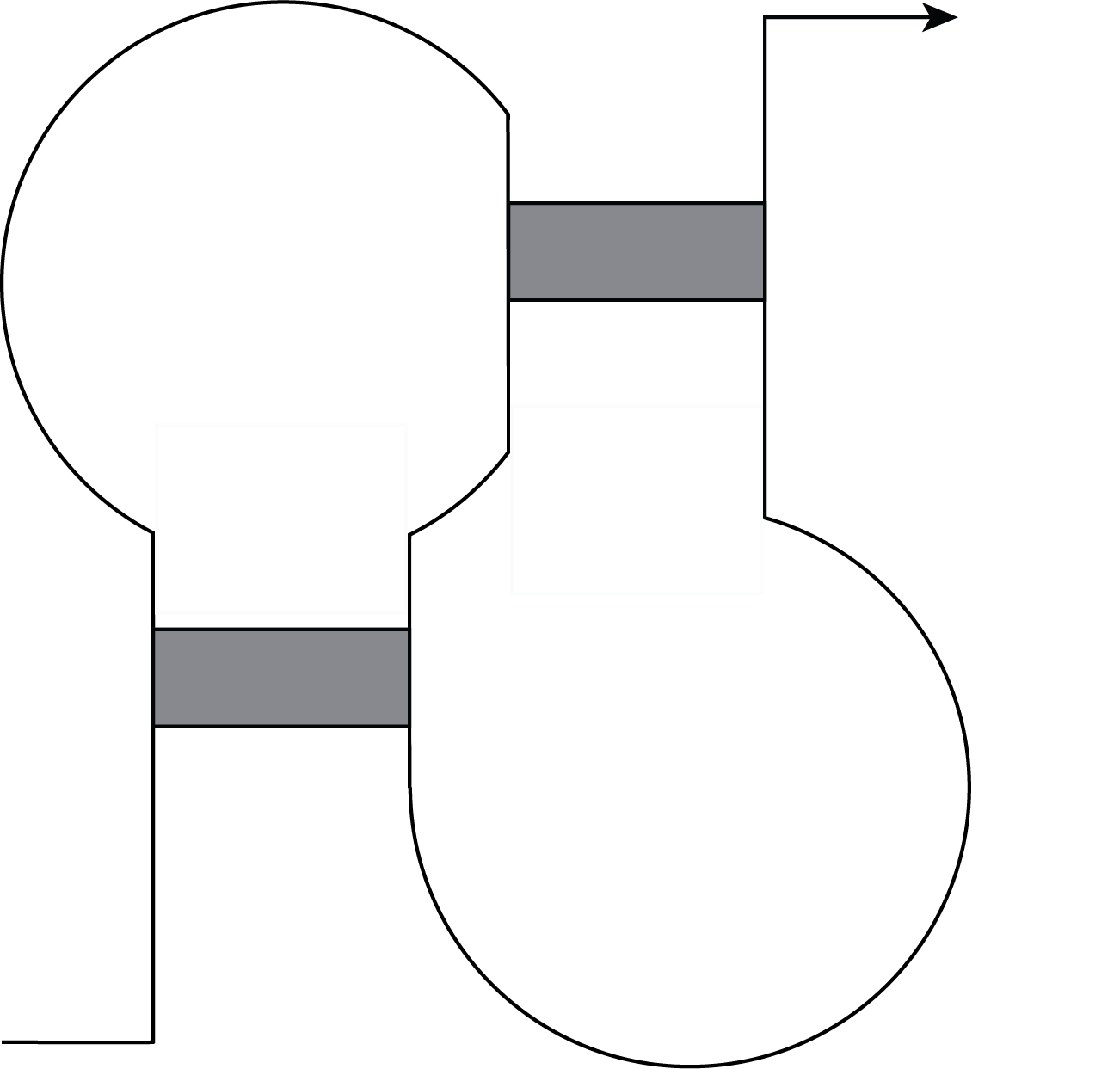}
\hspace{1cm}
        \begin{tikzpicture}[use Hobby shortcut]
\begin{knot}[
  consider self intersections = true,
  clip width=5,
  flip crossing/.list={3,5,6,7,8}
]
\strand ([closed]-2,1)..(-3,3)..(-1,4)..(1,2)..(-3,1)..(0,-1)..(1,0)..(-1,3)..(2,2)..(0,-2)..(-3,-1);
\draw[decoration={markings,mark=at position .04 with  {\arrow[scale=3,>=stealth]{>}}},postaction={decorate}]([closed]-2,1)..(-3,3)..(-1,4)..(1,2)..(-3,1)..(0,-1)..(1,0)..(-1,3)..(2,2)..(0,-2)..(-3,-1);
\end{knot}
\draw[line width=1.5mm,red] (-3,-.9)..(-2.45,.1);
\draw[line width=1.5mm,red] (-.5,1.32)..(.6,1.5);
\draw[latex'-latex', thick](-6.5,0)..(-5,0);
\node[above] at (-5.75,0) {\tiny $T$};
\node[left] at (-3,3) {\small $a$};
\node[left] at (-3,1) {\small $c$};
\node[left] at (-3,-1.5) {\small $d$};
\node[left] at (1.5,0) {\small $b$};
\end{tikzpicture}
\]
Using the corresponding stuck knot diagram we obtain the following coloring equations,

\begin{eqnarray*}
a &=& R_4(c,d) = c,\\
b &=& R_3(c,d) = d,\\
c &=& R_4(b,a) \ast a = b,\\
d &=& R_3(b,a) \ast a = a.
\end{eqnarray*}
From this system of equations we obtain $b=a$, $c=a$, $d=a$ for any $a\in \mathbb{Z}_4$, therefore, this system has 4 solutions in $\mathbb{Z}_4$. Thus, $Col_X \begin{pmatrix}
\includegraphics[scale=.15]{RNAFoldingEx.png}
\end{pmatrix}  = 4 $.  The stuck knot associated to each arc diagram above has a sticking number of $-2$. This means that the basic invariant cannot distinguish these two arc diagrams, but the stuquandle counting invariant can distinguish the two arc diagrams.

\end{example}

\section*{Acknowledgement} 
Mohamed Elhamdadi was partially supported by Simons Foundation collaboration grant 712462. The authors would like to thank Willi Kepplinger for his comments on an earlier version of this article.

\medskip

\bibliography{Ref}
\bibliographystyle{plain}

\end{document}